\documentclass[11pt,leqno]{amsart}

\usepackage{amsmath, amscd, amsthm, amssymb, graphics, mathrsfs, setspace, fancyhdr, times, bm, pdfsync, enumitem}
\usepackage[all]{xy}
\usepackage[usenames, dvipsnames, svgnames, table]{xcolor}

\usepackage[draft=false]{hyperref}
\hypersetup{breaklinks=true,colorlinks=true,linkcolor=black,anchorcolor=black,citecolor=black,filecolor=black,menucolor=black,urlcolor=MidnightBlue}
\usepackage{cleveref}

\numberwithin{equation}{section}

\newtheorem{thm}{Theorem}[subsection]
\newtheorem{prop}[thm]{Proposition}
\newtheorem{lm}[thm]{Lemma}
\newtheorem{lem}[thm]{Lemma}
\newtheorem{cor}[thm]{Corollary}

\theoremstyle{remark}
\newtheorem{rem}[thm]{Remark}

\newtheorem{ex}[thm]{Example}

\theoremstyle{definition}
\newtheorem{df}[thm]{Definition}
\newtheorem{defn}[thm]{Definition}
\newtheorem{num}[thm]{}

\numberwithin{equation}{thm}

\newtheorem{thm*}{Theorem}

\allowdisplaybreaks

\DeclareMathOperator{\SH}{SH}

\renewcommand{\AA}{\mathbb A}
\newcommand{\PP}{\mathbb P}
\newcommand{\Proj}{{\mathbb P}}
\newcommand{\HP}{\mathrm{H}\mathbb P}

\newcommand{\GGx}[1]{\mathbb G_{m,#1}}

\newcommand{\E}{\mathbb E}

\newcommand{\MSp}{\mathbf{MSp}}

\newcommand{\KO}{\mathbf{KO}}

\newcommand{\ZZ}{\mathbb{Z}}
\newcommand{\ZZe}{\mathbb{Z}_\epsilon}

\newcommand{\CC}{\mathbb{C}}

\newcommand{\NN}{\mathbb{N}}

\newcommand{\Sp}{\mathrm{Sp}}
\newcommand{\GL}{\mathrm{GL}}
\newcommand{\SL}{\mathrm{SL}}

\newcommand{\A}{\mathcal A}
\newcommand{\Laz}{\mathcal L}
\newcommand{\Bus}{\mathcal B}
\newcommand{\Walt}{\mathcal W}

\newcommand{\MForm}{\mathscr F} 
\newcommand{\GForm}{\mathscr G} 
\newcommand{\FGL}{\mathcal{FGL}}
\newcommand{\tFGL}{2-\mathcal{FGL}}
\newcommand{\FTL}{\mathcal{FTL}}

\newcommand*{\HMW}{\mathbf{H}_{\mathrm{MW}}}

\newcommand{\univ}{\mathrm{univ}}
\newcommand{\ev}{\mathrm{ev}}

\DeclareMathOperator{\GW}{GW}
\DeclareMathOperator{\Spec}{Spec}

\newcommand{\ux}{\underline x}
\newcommand{\uy}{\underline y}
\newcommand{\uc}{\underline c}
\newcommand{\ut}{\underline t}
\newcommand{\ue}{\underline e}

\newcommand{\ualp}{\underline \alpha}

\begin{document}

\title{Formal ternary laws and Buchstaber's 2-groups}

\author{David Coulette}
\address{ENS de Lyon, UMPA, UMR 5669, 46 all\'ee d'Italie, 69364 Lyon Cedex 07, France}
\email{david.coulette@ens-lyon.fr}

\author{Fr\'ed\'eric D\'eglise}
\address{ENS de Lyon, UMPA, UMR 5669, 46 all\'ee d'Italie, 69364 Lyon Cedex 07, France}
\email{frederic.deglise@ens-lyon.fr}
\urladdr{http://perso.ens-lyon.fr/frederic.deglise/}
 
\author{Jean Fasel}
\address{Institut Fourier - UMR 5582, Universit\'e Grenoble-Alpes, CS 40700, 38058 Grenoble Cedex 9, France}
\email{Jean.Fasel@univ-grenoble-alpes.fr}
\urladdr{https://www-fourier.univ-grenoble-alpes.fr/~faselj/}

\author{Jens Hornbostel}
\address{Fakult\"at f\"ur Mathematik und Naturwissenschaften, Bergische Universit\"at Wuppertal, Gau{\ss}str. 20, 42119 Wuppertal, Germany}
\email{hornbostel@math.uni-wuppertal.de}
\urladdr{http://www2.math.uni-wuppertal.de/~hornbost/}

\begin{abstract}
We compare formal ternary laws to Buchstaber's 2-valued formal group laws.
\end{abstract}

\maketitle
\tableofcontents

\section{Introduction}

\subsection{Background}
Formal group laws (or FGL for short) are fundamental in many areas of mathematics.
 In algebraic topology, since Quillen's work on complex cobordism,
 one dimensional commutative formal group laws have been an important tool
 when studying Chern classes for complex oriented cohomology theories $E^*$.
 In \cite{Quillen}, Quillen discovered that the first Chern class of a tensor product
 of line bundles is computed by a formal group law canonically associated with
 the $E^*$. Moreover, he showed that the formal group law associated with complex
 cobordism is the universal one, introduced by Lazard in \cite{Lazard55}.
 The work of Quillen has been hugely influential in algebraic topology.

 Among many of his pursuers, Buchstaber has intiatied a program for studying symplectic cobordism,
 via the realization of Pontryagin classes of sympectic bundles in complex cobordism
 (\cite{BN,Bu75, Bu76, Bu78}, see also \cite{AL75}).
 Most notably, he coined down the notion of (commutative one-dimensional) \emph{2-valued formal group}
 over a base ring $R$,
 formed by certain pairs $(F_1(x,y),F_2(x,y))$ of formal power series in two variables with coefficients in $R$
 (see Example \ref{ex:FGL&2FGL} for more details).
 In fact, he showed that the Pontryagin polynomial of a tensor product of the
 two copies of the universal symplectic bundle gives rise (after taking a ``square root'')
 to a 2-valued formal group. Moreover, the 2-valued group obtained in this way
 satisfies a special condition, that Buchstaber calls ``type I'' (see \Cref{df:Buchstaber_type}),
 and he finally shows that one obtains the universal 2-valued formal group of type I, at least after
 inverting $2$. This is therefore an analogue of Quillen's identification of the
 formal group law of complex cobordism with the universal one.

\medskip

More recently, Voevodsky opened a door 
 between algebraic topology and algebraic geometry, introducing among other things algebraic cobordism,
 algebraic Morava K-theory and the motivic Steenrod algebra.
 After a formalization of \emph{(algebraic) oriented theories} in algebraic geometry by Panin and Smirnov,
 Levine and Morel \cite{LM07} have extended the work of Quillen on cobordism and formal group laws,
 within the setting of oriented theories.
 In particular, they show that the algebraic cobordism of any characteristic $0$ field is isomorphic to the Lazard ring,
 and the formal group law derived from Chern classes in algebraic cobordism
 agrees with the universal formal group law: \emph{op. cit.} Th. 1.2.6.\footnote{This was later
 generalized by Marc Hoyois in \cite[Proposition 8.2]{Hoyois} for characteristic $p>0$ fields after inverting $p$,
 and by Markus Spitzweck in \cite[Theorem 6.7]{SpitCobord} for
 a local regular ring after inverting the characteristic exponent of its residue field.}
 The path opened by these mathematicians has become particularly fecund and inspired
 many new results and computations among which:
 Riemann-Roch formulas (\cite{PaninRR}, \cite{Smirnov}, \cite{Deg18}),
 motivic Landweber existence theorem (\cite{NSO}),
 Schubert calculus (\cite{HK11}, \cite{CPZ13}, \cite{CZZ}),
 operations on oriented theories (\cite{VishikOper}, \cite{MerkurjevOper}),
 bivariant formalism (\cite{SchYok}, \cite{Deg16}).

However, there are important theories in motivic homotopy,
 such as Chow-Witt groups (\cite{Barge00}, \cite{FaselCHW}) and higher Grothendieck-Witt groups
 (encompassing both orthogonal and hermitian K-theories, see e.g. \cite{Hornb}, \cite{ST} or \cite{Schlichting17})
 which are not oriented, while they possess a well-developed theory of characteristic classes
 (Euler and Thom classes).
 Motivated by these examples, Panin and Walter introduced a refinement of orientation theory in
 motivic homotopy theory in a series of papers \cite{Panin10pred, PWcobord, Panin19, Panin19b}.
 The idea of their theory is to restrict the existence of Thom classes to particular
 kind of vector bundles, the most important examples being the symplectic bundles
 ($\Sp$-orientations) and oriented bundles ($\SL$-orientations).
 The first case seems the most successful analogue of usual orientation (now $\GL$-orientation),
 as it allows to get a theory of \emph{Borel classes}, completely analogous to that
 of Chern classes.
 Borel classes are the motivic counter-part of Pontryagin classes in topology.

In the same spirit as Buchstaber's 2-valued groups,
 the second and third authors have developed in \cite{DF21} an initial
 idea of Walter which proposes to define an analogue for Borel classes of
 the formal group laws associated with Chern classes. This is the theory
 of \emph{formal ternary laws}. The ternary comes from the basic fact
 that a double tensor products of symplectic bundles is not symplectic, whereas
 a triple tensor product is. Therefore, formal ternary laws arise by looking at the
 Borel classes of triple tensor products of symplectic bundles of (minimal) ranks $2$.
 As such a triple tensor product is of rank $8$, there are $4$ relevant Borel classes
 which lead in fact to four power series in three variables
 (see \textsection\ref{sec:df_FTL}).

One interesting features of the theory is that one gets series with coefficients
 in the Grothendieck-Witt ring $\GW(\ZZ)$, rather than just $\ZZ$ in the classical case.
 Formal ternary laws are therefore more complicated than formal group laws,
 but still seem to play a classification role for motivic ring spectra analogous
 to the latter for usual ring spectra. In fact, the second and third authors
 successfully in \cite{DF21} computed the formal ternary law
 associated with Chow-Witt groups
 (and its higher analogue, the MW-motivic ring spectrum, \cite{BCDFO21})
 while that of Grothendieck-Witt groups was recently computed in \cite{FH21}
 -- see Example \ref{examples:ftl} for both cases.
 According to the nature of these theories, the corresponding formal ternary laws
 are the analogue of the additive and multiplicative formal group laws respectively.
 As in topology, these results are useful for computing stable operations.
 In particular the second and third authors built the so-called Borel character,
 analogue of the Chern character, which corresponds to the exponential isomorphism
 from the multiplicative to the additive formal group laws with rational coefficients.

\subsection{Main results}
The aim of this paper is first to develop the algebraic part of the theory
 of formal ternary laws with the aim to relate it with Buchstaber's
 2-valued formal groups.

In order to formulate this comparison, we introduce a generalization of Buchstaber's theory,
 the category $\GForm_{n,d}(R)$ of \emph{$n$-valued $d$-ary formal group laws} over a ring $R$,
 or $(n,d)$-groups for short (see \Cref{df:nd_groups}).
 The justification for this extension is
 that while 2-groups are $2$-valued secondary formal group laws,
 formal ternary laws not only are $4$-valued but also depends on three formal variables.
 In this general framework we introduce a natural notion of degree
 (\Cref{df:degree_nd_gps}), based on the analogy with homology.
 The minimal and maximal degrees of the coefficients of an $(n,d)$-group roughly measure its complexity.
 As an example, we could mention that there is only one formal group law with minimal and maximal degrees both $0$, the \emph{additive} one,
 formal group law of degree $1$ are exactly the \emph{multuplicative ones}, and there are no
 formal group of finite degree distinct from $0$ and $1$.

We also introduce the analogue of the Lazard ring (\Cref{prop:universal_nd-group})
 for $(n,d)$-groups.
 In particular, we let $\Bus$ be the ring which classifies the 2-groups, and $\Bus^0$ the quotient ring
 of $2$-groups whose coefficients have degree $0$. Then a basic but fundamental observation of Buchstaber
 can be restated (\Cref{prop:Bus_gradings}) by saying that:
$$
\Bus^0 \simeq \ZZ[\gamma]/\big((\gamma+2)(\gamma-2)^2\big).
$$
This is a fundamental difference with the structure of the Lazard ring, as this ring is in particular
 non-reduced.

Next we introduce after Walter the notion of \emph{$4$-valued formal ternary laws} (\emph{i.e.} FTL,
 see \Cref{df:FTL}),
 which is an algebraization of the notion appearing in \cite[2.3.2]{DF21}. Comparing
 to \emph{loc. cit.} and the original axioms of Walter, the novelty here is the so-called
 $\epsilon$-linearity axiom.
 We refer the reader to \Cref{num:geometricFTL} for the fact that these FTL
 are naturally attached to $\Sp$-oriented cohomology theories.\footnote{This improves
 the results of \emph{loc. cit.} which treated only the case
 of $\SL$-oriented theories.}

Note that FTL are not exactly $(4,3)$-groups: in fact, they have a \emph{quadratic} nature
 which is reflected by the fact that we work over the Grothendieck-Witt ring of $\ZZ$,
 $\GW(\ZZ) \simeq \ZZe=\ZZ[\epsilon]/(\epsilon^2-1)$ rather than just $\ZZ$ (see \Cref{rem:FTL_basics}).
 However we can introduce the classifying ring $\Walt$ of FTL, that we call
 the \emph{Walter ring}. Our first theorem is a classification result:
\begin{thm*}(see \Cref{thm:compute_universal_FTL})\label{thmi:compute_universal_FTL}
After inverting $2$, there are only two FTLs of degree $\leq 0$.
 One of this law is the FTL associated with Chow-Witt groups (see \cite{DF21}).
\end{thm*}
Therefore, the FTL associated with Chow-Witt groups is the analogue of the additive FTL.
 See \Cref{thm:compute_universal_FTL} for a more precise version in terms of the Walter ring
 and \Cref{rem:compute_universal_FTL} for further explanations.

Formal ternary laws have an extremely high computational complexity,
 due to the associativity axiom. More precisely, the Walter ring is naturally defined
 as a polynomial algebra with an infinite number of variables modulo an ideal $\mathcal R_{FTL}$,
 corresponding to the axioms (\Cref{prop:universal_FTL}).
 Nonetheless, this ideal is very difficult to compute.
 Therefore the preceding theorem could only be proved with the help of a computer.
 More generally, we developed a sage implementation which computes
 the Grobner basis of the ideal $\mathcal R_{FTL}$, for a possible range of degrees.
 See \Cref{num:sage} for more details.
%
%
%
%
%
%

The main result of this paper is the following comparison statement below between the following three categories over an arbitrary
commutative ring $R$:
\begin{itemize}
\item $\FGL(R)=\GForm_{1,2}(R)$ the category of formal group laws.
\item ($\tFGL)_I(R)\subset \GForm_{2,2}(R)$ Buchstaber's category of 2-groups of the first type.
\item $\FTL(R)=$ the category of formal ternary laws.
\end{itemize}
\begin{thm*}\label{thmi:comparison}
There is a functor $C: (\tFGL)_I \rightarrow \FTL$ which has a concrete interpretation in terms of characteristic classes. Composing it with the functor $N: \FGL \rightarrow (\tFGL)_I$ from Proposition \ref{prop:functorN} yields a functor which maps the FGL for Chow groups to the oriented part (set $\epsilon=-1$) of the FTL for Chow-Witt groups.
\end{thm*}

As a consequence of this theorem, it is in principle possible to determine all the FTL associated to oriented theories, including algebraic cobordism. It also gives an answer to the obvious question to know if FTL can be described in terms of more basic ingredients such as 2-valued formal groups \`a la Buchstaber. The answer is negative in general (e.g. \Cref{thm:HMWisnot2fgl} and the subsequent discussion). Concretely, this means that in general the Borel classes of a triple product of symplectic bundles can't be computed using hypothetic characteristic classes (satisfying Buchstaber's axioms) attached to a product of two symplectic bundles. 

To conclude, let us stress some consequences of the universality of the Walter ring $\Walt$. This discussion can be taken as a philosophical explanation of the results of the paper. If $\E$ is an Sp-oriented motivic ring spectrum over a field $k$, then there is a ring homomorphism
\[
\Walt\to \E^{2*,*}(\Spec k)
\]
defined by mapping the generators of the left-hand side to the corresponding coefficient in the formal ternary law attached to $\E$. In the particular case of $\MSp$, recall that there is at present no complete computation of $\MSp^{2*,*}(\Spec k)$, as well as of the homotopy groups of its topological counterpart $MSp$. Using complex realization, we obtain a string of ring homomorphisms
\[
\Walt\to \MSp^{2*,*}(\Spec {\mathbb{C}})\to \pi_{-2*}(MSp)
\]
which factors through a string of ring homomorphisms
\begin{equation}\label{eqn:complexrealization}
\Walt/(\epsilon+1)\to \MSp^{2*,*}(\Spec {\mathbb{C}})\to \pi_{-2*}(MSp)
\end{equation}
(see Remark \ref{rem:topologicalMsp} for more details). Using real realization instead, we obtain (using e.g. \cite[Corollary 4.7]{BH}) a sequence
\begin{equation}\label{eqn:realrealization}
\Walt/(\epsilon-1)\to \MSp^{2*,*}(\Spec {\mathbb{R}})\to \pi_{-*}(MU)
\end{equation}
and the ring $\Walt$ can then be thought of as providing a (possibly partial) set of generators and relations for both $\pi_{-2*}(MSp)$ and $\pi_{-*}(MU)$. At the moment, we are not able to perform integral computations in the Walter ring, and we leave this question to further work. After inverting $2$, the computations become much easier. For instance, Section \ref{sec:2inverted} shows that 
\[
(\Walt^{\leq 2}/\langle \epsilon-1\rangle)\otimes \ZZ[1/2]\simeq \ZZ[1/2][a^3_{111},a^1_{111}]/\langle (a^3_{111}-8)(a^3_{111}+8)\rangle
\]
which fits well with the computations of $\pi_{-*}(MU)\otimes \ZZ[1/2]$ (see Remark \ref{rem:compute_universal_FTL} for a tentative explanation of the relation $(a^3_{111}-8)(a^3_{111}+8)=0$). At present, we do not know if our system of relations in $\Walt$ is complete, i.e. that the ring homomorphism $\Walt\to  \MSp^{2*,*}(\Spec k)$ is an isomorphism for a reasonable field $k$. We hope to explore this question in future work.

\subsection{Plan of the paper}

In Section 2, we both recall and extend Buchstaber's theory of multi-valued series,
 by introducing possibly more formal variables leading to $(n,d)$-series and $(n,d)$-groups
 (resp. Def. \ref{df:multivalued} and \ref{df:nd_groups}).
 The main technical point is the notion of substitution for these generalized series
 (\ref{df:substitution}). We then recall Buchstaber's classification of $2$-FGL in our language.

Section 3 concerns the algebraic notion of formal ternary laws,
 based on the language of $(n,d)$-series in Section 1.
 We state our classification result (Theorem \ref{thmi:compute_universal_FTL})
 for low degree FTL, and give examples coming from motivic homotopy.

In Section 4, we establish in Theorem \ref{thmi:comparison} the comparison between FGL, $2$-FGL and FTL as stated above.

Finally in Section 5, we give the main steps for constructing the canonical FTL
 associated to an $\Sp$-oriented motivic ring spectrum.

\subsection*{Acknowledgments}
The first two authors thank David Monniaux and Bruno Salvy for their interest and help
 in the computational aspects of this paper.
 The third author wishes to thank Michel Brion and Adrien Dubouloz for useful conversations around the proof of Theorem \ref{thm:epsilonlinearity}, and Oliver R\"ondigs for discussions on faster than light travel. The authors wish to thank Achim Krause for pointing out he work of Kochman on the coefficient ring $MSp_*$ in topology. The last author's research was conducted in the framework of the research training group GRK 2240: {\em Algebro-geometric Methods in Algebra, Arithmetic and Topology}.




\section{On Buchstaber's multi-valued series and formal group laws}\label{sec:2FGL}

\subsection{Multivalued series}

The following definition is a variation on Buchstaber's algebraic theory
 of multi-valued power series, compare  \cite[section 1]{Bu75}.
\begin{df}\label{df:multivalued}
Let $R$ be a ring, and $n,d>0$ be integers.
 We let $\ux=(x_1,\hdots,x_d)$ be formal variables.
 An $n$-valued $d$-dimensional (power) series with coefficients in $R$, or \emph{$(n,d)$-series over $R$} for short,
 will be a polynomial
 $F_t(\ux)$ in $R[[\ux]][t]$ of degree $n$ and constant term $1$. We use the generic notation:
\begin{equation}\label{eq:power_generic_not1}
F_t(\ux)=1+F_1(\ux)t+\hdots+F_n(\ux)t^n.
\end{equation}
We will also consider a different way of writing this type of polynomials, that we call the
 \emph{Buchstaber form}:
$$
\tilde F_t(\ux):=t^nF_{(-t^{-1})}(\ux)=t^n-F_1(\ux)t^{n-1}+\hdots+(-1)^nF_n(\ux).
$$
We write $\MForm_{n,d}(R)$ the corresponding $R$-module.

Let $\varphi:R \rightarrow R'$ be a morphism of rings.
 The obvious corestriction of scalars functor $\varphi_*:R[[\ux]][t] \rightarrow R'[[\ux]][t]$ induces a map:
$$
\varphi_*:\MForm_{n,d}(R) \rightarrow \MForm_{n,d}(R'),
$$
which we will still call the \emph{corestriction} (along $R'/R$).
\end{df}

\begin{rem}\label{rem:nd-series}
\begin{enumerate}
\item Obviously, $\MForm_{n,d}(R)$ is isomorphic to the $R$-module $R[[\ux]]^n$
 and the $(n,d)$-series $F_t(\ux)$ is determined by the $n$-tuple of power series in the
 variables $\ux$: $F_1(\ux), \hdots, F_n(\ux)$.

In particular, $\MForm_{1,d}$ is isomorphic to the $R$-module of power series in $\ux$.
 Given such a power series $\varphi(\ux)$, we will write $\varphi_t(\ux)=1+\varphi(\ux)t$
 the corresponding object of $\MForm_{1,n}$.
\item The notation \eqref{eq:power_generic_not1} is motivated by the theory of characteristic classes,
 and especially total classes (Chern, Pontryagin, Borel, etc.).
 The Buchstaber form $\tilde F_t(\ux)$, taken from the convention of \cite{Bu75},
 is useful when reasoning in terms of roots. We keep the two possible conventions to help the reader
 compare our notations with that of Buchstaber, but we believe that the notation \eqref{eq:power_generic_not1}
 is better suited when dealing with examples from motivic homotopy theory which are our motivation.
\end{enumerate}
\end{rem}

\begin{ex}\label{ex:total_chern}
Let $\E$ be an oriented motivic ring spectrum over a scheme $S$ and let $\PP_S^{\infty}$ be the infinite projective space (which is the colimit of $\PP_S^n$ for $n$ increasing). Let $\E^{**}=\E^{**}(S)$  be the ring of coefficients.

Then it is well-known that $\E^{**}(\PP^{\infty}_S)=\E^{**}[[x]]$,
 where $x$ is the first Chern class of the universal line bundle on $\PP^{\infty}_S$. Note further that $x$ is of cohomoloical degree $(2,1)$. More generally, we have $\E^{*,*}((\PP^{\infty}_S)^{\times d})=\E^{**}[[x_1,\ldots,x_d]]$ for any $d\geq 1$, where $x_i$ are the first Chern classes of the corresponding line bundles, sitting in degree $(2,1)$. Let now $G_{d,\infty}$ be the Grassmannian of $d$-dimensional sub-vector bundles in an infinite dimensional vector bundle over $S$,
 and let $V_{d,\infty}$ be the universal $d$-dimensional vector bundle over $\mathrm{Gr}_{d,\infty}$. The sum of the universal line bundles yields a morphism 
\[
 (\PP^{\infty}_S)^{\times d}\to G_{d,\infty}
\]
and we may consider the pull-back of the total Chern class (with coefficients in $\E^{*,*}(G_{d,\infty})$)
\[
c_t(V):=1+\sum_{i=1}^d c_i(V)t^i,
\]
along this morphism. Each $ c_i(V)$ is a homogeneous power series of degree $i$ in the variables $x_1,\ldots,x_d$ and we obtain an element of $\MForm_{d,d}(\E^{*,*})$. More precisely, we have $c_i(V)=e_i(x_1,\ldots,x_d)$, the elementary symmetric polynomial of degree $i$ in the variables $\ux$. Further, consider the map
\[
 (\PP^{\infty}_S)^{\times 2}\to  \PP_S^{\infty}
\]
classifying the tensor product of the two universal line bundles. Pulling-back the Chern class $1+xt$ to $(\PP^{\infty}_S)^{\times 2}$, we obtain a polynomial in $\MForm_{1,2}(\E^{*,*})$ of the form
\[
1+F(x_1,x_2)t
\] 
where $F(x_1,x_2)$ is the first Chern class of the product. This is the so-called formal group law of $\E^{*,*}$.
\end{ex}

\begin{num}
We can generically write an $(n,d)$-series $F_t(\ux)$ with coefficients in $R$ as
\begin{equation}\label{eq:generic_writing}
F_t(\ux)=1+\sum_ {\ualp \in \NN^n, \\ l=1,\hdots,n} a_{\ualp}^l\ux^{\ualp}t^l,
\end{equation}
for coefficients $a_{\ualp}^l \in R$. The preceding example leads us to consider a grading on $(n,d)$-series,
 defined by the convention that the formal variables $x_1,...,x_d$ have degree $+1$ and the indeterminate $t$ has degree $-1$.
\end{num}
\begin{df}\label{df:degree_nd_gps}
Consider the above notation. We define the degree (resp. valuation) of the $(n,d)$-series $F_t(\ux)$ as the integer $\delta$ (resp. $\nu$)
 such that:
\begin{align*}
&\delta=\sup(A), \text{resp. } \nu=\inf(A), \\
&A=\{|\ualp|-l \mid \ualp \in \NN^n, 1 \leq l \leq n, a_{\ualp}^l \neq 0 \}.
\end{align*}
In particular, we view $\MForm_{n,d}(R)$ as a $\ZZ$-graded $R$-module.
%
\end{df}
By definition, $\delta \geq \nu\geq -n$. The interval $[\nu,\delta]$ measures the
 complexity of the $(n,d)$-series $F_t(\ux)$.

\begin{ex}
In Example \ref{ex:total_chern}, one notes that the Chern classes $c_l(V)$ are of bidegree $(2l,l)=l(2,1)$.
 In particular, if one writes:
$$
c_l(V)=\sum_{\ualp} a_{\ualp}^l\uc^{\ualp},
$$
the (cohomological) bidegree of $a_{\ualp}^l$ is $(l-|\ualp|)(2,1)$, and this legitimates the choice of grading
 for $\uc^{\ualp}t^l$ to be $|\ualp|-l$. In that particular case, we have $\nu=\delta=0$.
\end{ex}

\begin{rem}
There is an obvious multiplication map:
$$
\MForm_{n,d}(R) \times \MForm_{m,e}(R) \rightarrow \MForm_{n+m,d+e}, F_t(\ux),G_t(\uy) \mapsto F_t(\ux).F_t(\uy)
$$
where the product is taken in the ring $R[[\ux,\uy]][t]$. Whereas this product is quite natural,
 we will rather use the substitution product of \Cref{df:substitution} below.
\end{rem}

\begin{df}\label{df:roots_nd-series}
Consider an $(n,d)$-series $F_t(\ux) \in \MForm_{n,d}(R)$.
 A \emph{root} of $F_t(\ux)$ will be a root of $\tilde F_t(\ux)$, as a polynomial in $t$.
 We say that $F_t(\ux)$ is \emph{split} if $\tilde F_t(\ux)$ is split, as a polynomial in $t$.
In the latter case, we will generically denote by $F^{[i]}$, $i=1,...,n$ the corresponding roots.
 It follows that
$$
F_t(\ux)=\prod_{i=1}^n (1+F^{[i]}t).
$$
\end{df}

\begin{ex} The previous definition is motivated by the splitting principle. As in Example \ref{ex:total_chern}, we have 
\[
c_t(V):=1+\sum_{i=1}^d c_i(V)t^i=\prod_{i=1}^d(1+x_it)
\]
In other words, the $n$-dimensional $d$-valued series $c_t(V)$ is split in $\E^{*,*}[[\ux]][t]$ -- \emph{i.e.}
 the polynomial $\tilde c_t(V) \in \E^{**}[[\uc]][t]$ splits.
\end{ex}

\begin{rem}
Beware that if $R$ is not integral, the roots of a split series $F_t(\ux)$ are not unique
 in general.
\end{rem}

\begin{num}\textit{The algebraic splitting principle}.\label{num:alg_split_pple}
Consider an $(n,d)$-series $F_t(\ux)$ with coefficients in $R$.
 Note that there always exists a finite ring extension $R[[\ux]] \rightarrow R^\prime$
 such that $\tilde F_x(\ux)$ splits in $R^\prime$. In fact, there is a universal one given by the formula
 \[
R^\prime=R[[\ux]][\ut]/(e_i(\ut)-F_i(\ux), i=1,\hdots,n)
 \] 
 where $\ut=(t_1,...,t_d)$, $e_i(\ut)$ is the $i$-th elementary symmetric polynomial in the variables
 $\ut$. Given a finite extension $A/R[[\ux]]$ in which $\tilde F_x(\ux)$ splits, we denote by $F^{[i]}$ (or sometimes $F_A^{[i]}$ if $A$ is important) its roots so that we have:
$$
F_t(\ux)=\prod_{i=1}^n (1+F^{[i]}t),
$$
and call them the \emph{formal roots} of $F_t(\ux)$ in $A$. There is a bijection between the set of pairs $(A,F^{[i]}_A)$ (with $A$ finite over $R[[\ux]]$) and the set of finite homomorphisms of $R[[\ux]]$-algebras $R^\prime\to A$. 
\end{num}

\begin{num}\textit{Substitutions}.
We will say that a power series $\varphi(x)$ in one variable is \emph{composable} if $\varphi(0)=0$. Then for any power series $\psi(x)$, the composite $\psi(\varphi(x))$
 is well-defined. Similar considerations hold when dealing with more variables and we now extend them in the context of multivalued series.
 
Let then $F_t(\ux) \in \MForm_{n,d}$ with $\ux=(x_1,\hdots,x_d)$.
 We fix an integer $1 \leq i \leq d$ and consider the $(d-1)$-tuple of variables
 $\ux'=(x_1,...,x_{i-1},x_{i+1},...,x_d)$.
  We introduce another $r$-tuple of variables $\uy=(y_1,...,y_r)$ and consider the expression in $R[[\ux',\uy]][t]$:
$$
\phi_i(\ux',\uy)=\prod_{l=1}^r F_t(x_1,...,x_{i-1},y_l,x_{i+1},...,x_d).
$$
Note that this is a polynomial in $t$ of degree $rn$ and with constant term $1$.
 Moreover, $\phi_i(\ux',\uy)$ is invariant under the action of the group $\mathfrak S(\uy)$ of permutations of the variables $\uy$.
 In particular, there exists a unique element $\phi_i' \in R[[\ux',\ue]][t]$, where $\ue=(e_1,...,e_r)$,
 such that
$$
\phi_i(\ux,\uy)=\phi'_i(\ux,e_l=e_l(\uy), l=1,\hdots,r),
$$
where the right-hand side is obtained by the indicated substitution, and $e_l(\uy)$ denotes the $l$-th symmetric
 elementary polynomial in the variables $\uy$.
 By construction, $\phi_i'(\ux',\ue)$ is polynomial in the variable $t$, of degree $rn$,
 with coefficients formal series in the variables $\ux'$ and $\ue$. In fact, $\phi_i'$
 is an $(rn,d+r-1)$-series that will serve for the following definition:
\end{num}
\begin{df}\label{df:substitution}
We will say that the multivalued series $G_t(\uy)\in \MForm_{m,r}$
 is \emph{composable} if $G_t(\underline 0)=1$. More precisely, if we write
\[
G_t(\uy)=1+\sum_{l=1}^m G_l(\uy)t^l.
\]
then $G_t(\uy)$ is composable if $G_l(\underline 0)=0$ for $l=1,\ldots, m$. In that case, one defines the
 \emph{substitution of $G_t(\uy)$ in $F_t(\ux)\in \MForm_{n,d}$ at the place $i$}
 as the following $(nm,d+r-1)$-series with coefficients in $R$:

$$
F_t\big(x_1,...,x_{i-1},G_t(\uy),x_{i+1},...,x_d\big)
 :=\phi_i'\big(\ux,e_l=G_l(\uy), l=1,\hdots,r\big),
$$
which is well-defined as by assumption, each formal series $G_l(\uy)$ is composable.

Let us write $\MForm_{n,d}^\circ$ the $R$-module of $(n,d)$-series which are composable.
 Then we have defined a bilinear pairing of $R$-modules for any $1 \leq i\leq n$:
$$
\sigma_i:\MForm_{n,d}^\circ \times \MForm_{m,e}^\circ \rightarrow \MForm_{nm,d+e-1}^\circ
$$
\end{df}
Note that with this definition, a power series $\varphi(\ux)$
 is composable if an only if the associated $(1,d)$-series $\varphi_t(\ux)$ 
 -- Remark \ref{rem:nd-series}(1) -- is composable.

\begin{rem}
Following Buchstaber, one can present the previous definition
 using the algebraic splitting principle.
 Indeed, if one introduces the formal roots $G^{[l]}$ of $G_t(\uy)$,
 we can write suggestively:
$$
F_t\big(x_1,...,x_{i-1},G_t(\uy),x_{i+1},...,x_d\big)
 =\prod_{l=1}^m F_t\big(x_1,...,x_{i-1},G^{[l]},x_{i+1},...,x_d\big).
$$
Note that $G^{[l]}$ is in general not a power series in the variables
 $\uy$.
\end{rem}

As a first application of this general concept,
 one introduces morphisms of multivalued series.
\begin{df}\label{df:morphism_nd-series}
Let $F_t(\ux)$ and $G_t(\ux)$ be composable $(n,d)$-series
 with coefficients in $R$.
 A morphism from $F_t(\ux)$ to $G_t(\ux)$ is a composable power series
 $\Theta(x) \in R[[x]]$ such that (in $R[[\ux]][t]$):
$$
\Theta_t(F_t(x_1,\hdots,x_d))=G_t(\Theta(x_1),\hdots,\Theta(x_d)).
$$
where $\Theta_t(x)$ is the $(1,d)$-series associated with $\Theta(x)$
 (cf. \Cref{rem:nd-series}(1)).

The composition of such morphisms is defined by the composition
 of power series. So in particular, $\MForm_{n,d}^\circ(R)$ becomes
 a category, and in fact an $R$-linear category.

It follows from this definition that a morphism $\Theta(x)$
from $F_t(\ux)$ to $G_t(\ux)$ is an isomorphism if and only if
$\Theta(x)=\lambda.x+O(x^2)$ where $\lambda \in R^\times$ (e.g. \cite[Lemma 2.8]{Strickl}).
A usual convention in this setting is to say that $\Theta(x)$
is a \emph{strict isomorphism} when $\lambda=1$.
\end{df}

\begin{rem}
Explicitly, the equation in the above definition reads:
\[
\prod_{l=1}^n(1+\Theta(F^{[l]})t)=G_t(\Theta(x_1),\hdots,\Theta(x_d)).
\]
If we denote by $G_{\theta}^{[l]}$ the roots of the series $G_t(\Theta(x_1),\hdots,\Theta(x_d))$, we may also write suggestively $\Theta(F^{[l]})=G_{\theta}^{[l]}$.
\end{rem}

\begin{num}
It is clear that given a morphism of rings $\varphi:R \rightarrow R'$
 the application $\varphi_*:\MForm_{n,d}^\circ(R) \rightarrow \MForm_{n,d}^\circ(R')$
 is in fact a functor of $\ZZ$-linear categories,
 so that we get a (strict) covariant functor $R \mapsto \MForm_{n,d}^\circ(R)$
 from rings to $\ZZ$-linear categories. Using the Grothendieck
 construction, this can be integrated. Explicitly, we get the following definition.
\end{num}
\begin{df}\label{df:nd-series_category}
We let $\MForm_{n,d}^\circ$ be the category whose objects are pairs $(R,F_t(\ux))$
 where $R$ is a ring and $F_t(\ux)$ is a composable $(n,d)$-series with coefficients in $R$,
 and whose morphisms $(R,F_t(\ux))) \rightarrow (R',G_t(\ux)))$ are pairs
 $(\varphi,\Theta)$ where $\varphi:R \rightarrow R'$ is a ring homomorphism 
 and $\Theta:\varphi_*(F_t(\ux)) \rightarrow G_t(\ux)$ a morphism of $(n,d)$-series
 with coefficients in $R'$. Note that 
 \[
 (\varphi,\Theta)=(\mathrm{Id},\Theta)\circ (\varphi,\mathrm{Id})
 \]
With this definition, $\MForm_{n,d}^\circ$ is a $\ZZ$-linear category, cofibered over the category of rings.
\end{df}

\subsection{Multi-valued formal group laws}

The following definition is a generalization
of Buchstaber's definition which corresponds to $d=2$, see \cite[Definition 1.2]{Bu75}.
\begin{df}\label{df:nd_groups}
Let $R$ be a ring, $n>0$, $d>1$ integers, and $\ux=(x_1,...,x_d)$ formal variables.
 A $d$-ary $n$-valued formal group law, $(n,d)$-group for short,
 is an $(n,d)$-series $F_t(\ux)$ satisfying the following properties:
\begin{enumerate}
\item \textit{Neutral element}. The $(n,1)$-series $F_t(x,0,...,0)$ is split
 with one root $x$ of multiplicity $n$, \emph{i.e.} $F_t(x,0,...,0)=(1+xt)^n$.
\item \textit{Symmetry}. The element $F_t(\ux)$ of $R[[\ux]][t]$ is a fixed
 point under the action of the group $\mathfrak S(\ux)$ permuting the formal variables.
\item \textit{Associativity}. Given another $(d-1)$-tuple of formal variables
 $(x_{d+1},...,x_{2d-1})$, one has the following equality of $(n^2,2d-1)$-series:
\begin{equation}
\begin{split}
&F_t\big(F_t(x_1,\hdots,x_d),x_{d+1},\hdots,x_{2d-1}\big) \\
&\quad =F_t\big(x_1,F_t(x_2,\hdots,x_{d+1}),x_{d+2},\hdots,x_{2d-1}\big)
\end{split}
\end{equation}
where we have used the substitution operation (\Cref{df:substitution}),
 valid as by property (1) $F_t(\ux)$ is composable.
\end{enumerate}
We will denote by $\GForm_{n,d}(R)$ the full sub-category of $\MForm_{n,d}^\circ(R)$ defined
 by $(n,d)$-groups. The corestriction of scalars functor
 obviously respects $(n,d)$-groups so that we can also consider
 the sub-category $\GForm_{n,d}$ of $\MForm_{n,d}$, cofibered over the category of rings,
formed by pairs $(R,F_t(\ux))$ such that $F_t(\ux)$ is an $(n,d)$-group.
\end{df}

\begin{ex}\label{ex:FGL&2FGL}
Our two main examples are the following. The above more general definition will be useful in the next section (see \Cref{df:FTL}).
\begin{itemize}
\item $\FGL=\GForm_{1,2}$ is the category of formal group laws.
\item $\tFGL=\GForm_{2,2}$ is Buchstaber's category of 2-valued formal groups (using the terminology of \cite{BN, Bu75, AL75}). 
\end{itemize}
Indeed, let 
\[
\tilde F_t(x,y)=t^2-F_1(x,y)t+F_2(x,y)
\]
and suppose that it satisfies Buchstaber's axioms. We then have 
\[
F_t(x,y)=1+F_1(x,y)t+F_2(x,y)t^2
\]
and a direct inspection shows that the symmetry and neutral element axioms are equivalent. 
Our associativity axiom reads as 
\[
F_t(F_t(x,y),z)=F_t(x,F_t(y,z))
\]
If $F^{[1]}$ and $F^{[2]}$ are the roots of $F$, then this reads as
\[
F_t(F^{[1]},z)F_t(F^{[2]},z)=F_t(x,F^{[1]})F_t(x,F^{[2]}),
\]
which translates apparently into Buchstaber's four constraints, see \cite[section 2]{Bu75} or \cite[p. 299]{AL75}. 
\end{ex}

\begin{rem}
Recall that a formal group law is a commutative group object in one-
dimensional formal schemes together with the choice of a local parameter. There
are several possible definitions of commutative $d$-ary monoids or groups, and of
multivalued variants of these. None of these will lead to the above definition of
$n$-valued $d$-ary laws. First, we did not discuss the notion of
inverses here. (Recall that for formal group laws, their existence is automatic, see
e.g. \cite[Proposition A2.1.2]{Ra03}. Second, the associativity conditions are not separate conditions on the $n$ different values $F_i$, but something more complicated, enforced by the associativity condition for the ``variable'' $F(x, y)$.
\end{rem}

\begin{rem}\label{rem:noepsilon}
\begin{enumerate}
\item Recall \cite[Theorem 6.1]{Ha78} that one-dimensional formal group laws are commutative over reasonable rings. In general, it is easy to describe non necessarily commutative multivalued formal group laws. 
\item Using symplectic bundles and $MU^*(\HP^{\infty})$, one obtains the motivating example for $\tFGL=\GForm_{2,2}$, see e.g. \cite[section 2]{AL75}. A similar construction can be done in the motivic case.  
\item It is possible to refine the above notion using the ring $\ZZ_{\epsilon}=\ZZ[\epsilon]/(\epsilon^2-1)$ introduced in Section \ref{sec:FTL}, and obtain in particular the notion of an ``unoriented'' two-valued formal group law. Such an object is a $(2,2)$-series $F_t(x,y)$ satisfying the symmetry and associativity axioms, together with the following refined neutral element axiom: The $(2,1)$-series $F_t(x,0)$ is split with roots $x$ and $-\epsilon x$,
and the additional condition $F_t(-\epsilon x,y)=F_{-\epsilon t}(x,y)$. Lacking interesting examples of genuine such objects (i.e. on which $\epsilon$ doesn't act as $-1$), we refrain from pursuing this here in detail, but see Proposition \ref{thm:HMWisnot2fgl} below.
\end{enumerate}
\end{rem}


The following result is an obvious generalization of known cases.
\begin{prop}\label{prop:universal_nd-group}
For any couple of integers $(n,d)$, the category $\GForm_{n,d}$
 admits an initial object $(\A_{n,d},F_t^{\univ,n,d})$.

If one writes
$$
F_t^{\univ,n,d}(\ux)=1+\sum_ {\ualp \in \NN^d, \\ l=1,\hdots,n} a_{\ualp}^l\ux^{\ualp}t^l,
$$
then the $\ZZ$-algebra $\A_{n,d}$ is generated by the family $\{a_{\ualp}^l, \ualp \in \NN^d, 0 \leq l \leq n\}$.
 It is a finitely presented algebra:
$$
\A_{n,d}=\ZZ\big[ a_{\ualp}^l, (\ualp,l) \in \NN^d \times \{0,\hdots,n\}\big]/\mathcal R_{n,d}
$$
and the ideal of relations $\mathcal R_{n,d}$ is generated by homogeneous elements of the graded polynomial ring
 $\ZZ\big[ a_{\ualp}^l\big]$ with respect to the degree defined by 
$$
\deg(a_{\ualp}^l)=|\ualp|-l.
$$
\end{prop}
\begin{proof}
The proof is well-known: to get the ring $\A_{n,d}$, take the polynomial over $\ZZ$ generated
by the variables $a_{\ualp}^l$ as above, and mod out by the relations coming from the axioms of $(n,d)$-groups.
Then, define $F_t^{\univ}(\ux)$ by the above formula.
It is an easy (but lengthy) verification that the relations generating the ideal $\mathcal R_{n,d}$
are homogeneous with respect to this notion of degree. In particular, the ring $\A_{n,d}$ is naturally
$\ZZ$-graded.
\end{proof}

\begin{df}\label{df:universal_nd_rings}
We call the graded ring $\A_{n,d}$ defined above the universal ring of $(n,d)$-groups,
 and $F_t^{\univ,n,d}(\ux)$ --- or simply $F_t^{\univ}(\ux)$ when $(n,d)$ is clear --- the universal $(n,d)$-group.

Obviously, $\Laz=\A_{1,2}$ is the Lazard ring. We put $\Bus=\A_{2,2}$ and call it the \emph{Buchstaber ring}. 
\end{df}

\begin{ex}
The famous theorem of Lazard (\cite[Th\'eor\`eme II]{Lazard55}) asserts that $\Laz$ is isomorphic to a polynomial ring $\ZZ[u_1,u_2,...]$
 with an infinite number of variables $u_i$.
\end{ex}

\begin{num}
From the above definition, the Buchstaber ring $\Bus$ is generated by elements of the form $a_{ij}^l$,
 $l=1,2$, $i,j \geq 0$, such that the universal $2$-FGL is
$$
F_t^{\univ}(x,y)=1+\sum_{i,j} a_{ij}^1x^iy^jt+\sum_{i,j} a_{ij}^2x^iy^jt^2.
$$
Specializing \Cref{df:nd_groups} to $n=d=2$, 
 the symmetry axiom tell us exactly that $a_{ij}^l=a_{ji}^l$, and the neutral axiom gives:
\begin{align*}
& F_t^{\univ}(x,0)=(1+tx)^2=1+2tx+t^2x^2, \\
\emph{i.e.}\  & a^1_{i0}=2 \delta_1^i, \ a^2_{i0}=\delta_2^i.
\end{align*}
In particular, all elements of $\Bus$ must have non-negative degree, and the only element
 of degree $0$ which is not determined is $a^2_{11}$. Buchstaber denotes this element by $\gamma$ and essentially observed
 the following result in \cite[Theorem 2.3]{Bu75}.\footnote{Compared to Buchstaber original reference,
 we only add the last observation of the following statement.}
\end{num}
\begin{prop}\label{prop:Bus_gradings}
The Buchstaber ring $\Bus$ defined above is an $\NN$-graded ring \emph{i.e.} each element has non-negative
 degree.
 The sub-ring $\Bus^0$ of elements of degree $0$ satisfies:
$$
\Bus^0 \simeq \ZZ[\gamma]/\big((\gamma+2)(\gamma-2)^2\big).
$$
Note in particular that $\Bus$ is an augmented $\Bus^0$-algebra.
\end{prop}

\begin{proof}
The first assertion has already been observed. The second one is now a consequence of the fact
 that the associativity relations which are of degree $0$ reduce to the single relation
$(\gamma+2)(\gamma-2)^2=0$, and this concludes.
\end{proof}

If we work over $\ZZe$ (compare Remark \ref{rem:noepsilon}), the above computation in degree zero refines to $F^{univ,0}_{\epsilon}$ given by 
\[
(F^{univ,0}_{\epsilon})_t(x,y)=1+(1-\epsilon)(x+y)t+\left(-\epsilon(x^2 + y^2) + a xy\right)t^2
\]
over $\Bus^0_{\epsilon} \simeq \ZZe[a]/\langle (1 + \epsilon)a, (a - (1 - \epsilon))^2 (a + (1 - \epsilon))\rangle$.


\begin{ex}
It follows in particular that all $2$-formal groups $F_t(x,y)$ whose degree
 (\Cref{df:degree_nd_gps}) is equal to $0$ are of the form:
$$
F_t^0(x,y)=1+2(x+y)t+(x^2+\gamma xy+y^2)t^2
$$
with either $\gamma=-2$ or $\gamma=2$.
One can rephrase this by saying that $(\Bus^0,F_t^0)$ is the universal $2$-formal group of degree $0$.
Buchstaber calls the above an \emph{elementary 2-formal group}. 
\end{ex}

Observe that if $\pi:\Bus\to \Bus^0$ is the ring homomorphism obtained by killing the variables $a^1_{ij}$ and $a^2_{ij}$ of (strictly) positive degree, then we obtain a morphism
\[
(\pi,\mathrm{Id}):(\Bus,F^{\univ}_t)\to (\Bus^0,F_t^0)
\]
in $\GForm_{2,2}$, which is just the morphism given by the universal property of $(\Bus,F^{\univ}_t)$. On the other hand, the ring homomorphism $\Bus^0\to \Bus$ cannot be extended to a morphism in $\GForm_{2,2}$.

\begin{rem}
Note in particular a huge difference between the Lazard ring and the Buchstaber ring: $\Bus$ is non reduced. From the above computation $(\gamma+2)(\gamma-2)$ is a non-zero nilpotent element. In particular, $\Bus$ cannot be a polynomial ring.
\end{rem}

We close this section with the following elementary lemma (compare \cite[Corollary 2.4]{Bu75}).
\begin{lm}
There is a canonical monomorphism of rings
$$
\Bus^0 \rightarrow \big(\ZZ[\gamma]/(\gamma+2)^2\big) \times \big(\ZZ[\gamma]/(\gamma-2)\big)
$$
which is an isomorphism after inverting $2$.
Moreover, there is canonical monomorphism of rings
$$
\Bus^0_{red} \rightarrow \big(\ZZ[\gamma]/(\gamma+2)\big) \times \big(\ZZ[\gamma]/(\gamma-2)\big)
$$
which is an isomorphism after inverting $2$.
\end{lm}

In other words, $\Bus^0_{red}$ is isomorphic to $\ZZ\times \ZZ$,  such
 that the projections to the first and second factors are respectively the evaluations at $\gamma=-2$ and $\gamma=+2$.

\subsection{Buchstaber's classification of 2-formal groups}

Given the computation of the degree $0$ part of the Buchstaber ring, and following \cite{Bu75}, one sets the following notation.
\begin{df}\label{df:Buchstaber_type}
A $2$-formal group $F_t(x,y)$ over a ring $R$ will be said to be \emph{of type I} (resp. \emph{type II})
 if the coefficient of $xyt^2$, generically denoted by $\gamma=a_{11}^2$, is equal to $-2$
 resp. $+2$.
 Further, we define the two following rings:
$$
\Bus_I=\Bus/(\gamma+2), \quad \Bus_{II}=\Bus/(\gamma-2),
$$
\end{df}
According to what was said before,
 the ring $\Bus_I$ classifies the $2$-formal groups of type I.
 As, according to our conventions, $\gamma=a_{11}^2$ is of pure degree $0$,
 these two rings are again $\NN$-graded (\Cref{prop:Bus_gradings}). We write $(\tFGL)_I$ for the full subcategory of $\tFGL$ whose objects are laws of the first type and $(\tFGL)_{II}$ for the full subcategory of objects of the second type. 

The following statement is essentially an elaboration of Buchstaber's original results.

\begin{prop}
Let $R$ be a ring. The following assignment
$$
\sigma_R:\FGL(R) \rightarrow \tFGL(R), F(x,y) \mapsto F^2_t(x,y)=(1+F(x,y)t)^2
$$
defines a faithful functor, by taking the identity on morphisms.
 Its essential image lies in $2$-formal groups
 which are split (\Cref{df:roots_nd-series}) and of the second type.

If $2 \in R^\times$ then $\sigma_R$ is full.
 If in addition $R$ is a domain, then $\sigma_R$ defines an equivalence of categories
 between formal group laws over $R$ and the $2$-formal groups over $R$ which are
 of the second type.
\end{prop}

\begin{proof}
The $(2,2)$-series $F^2_t(x,y)$ obviously satisfies the neutral and symmetry axioms of \Cref{df:nd_groups},
 as $F(x,y)$ satisfies the corresponding axioms.
 As it is moreover splits, with a double root $F(x,y)$, one easily deduces the associativity axiom of $2$-formal groups
 from the associativity axiom of $F(x,y)$.

Given a morphism $\Theta:F(x,y) \rightarrow G(x,y)$, we get:
\begin{equation}\label{eq:morphism}
\begin{split}
&\Theta_t(F^2_t(x,y))=(1+\Theta(F(x,y))t)(1+\Theta(F(x,y))t) \\
&=(1+G(\Theta(x),\Theta(y)))t)(1+G(\Theta(x),\Theta(y)))t)
 =G_t^2(\Theta(x),\Theta(y)).
\end{split}
\end{equation}
where the first equality holds as $F_t^2(x,y)$ splits (with double root $F(x,y)$), and the second one as $\Theta(x)$ is a morphism
 of FGL. In particular, $\Theta_t$ is indeed a morphism of $2$-FGL.

The assertion on the image of $\sigma_R$ is obvious as $F(x,y) \equiv x+y\ (xy)$.
 Let us assume that $2 \in R^\times$. We consider a morphism $\Theta_t=1+\Theta(x)t$
 from $F^2_t(x,y)$ to $G_t^2(x,y)$ and we check that $\Theta(x)$ is a morphism of FGL from $F$ to $G$.
 By assumption, we know that Equation \eqref{eq:morphism} holds.
 Expanding and taking the coefficient of $t$, one deduces that
$$
2\Theta(F(x,y))=2G(\Theta(x),\Theta(y))
$$
which allows to conclude.

The final assertion was already observed by Buchstaber. If $R$ is integral,
 the roots $\theta_1(x,y)$ and $\theta_2(x,y)$ of a split $2$-formal group $F_t(x,y)$ are unique (as roots of the polynomial $\tilde F_t$ in $t$,
 with coefficients in the integral ring $R[[x,y]]$).
 One deduces from the associativity axiom of $2$-formal groups that
 $\theta_1(x,y)=\theta_2(x,y)$ and that is satisfies the axioms of a formal group law.
 This implies the essential surjectivity of $\sigma_R$.
\end{proof}

%

\begin{rem}
The recipe to make the morphism $\mathfrak S$ explicit is as follows. We let $F^{\univ}(x,y)=x+y+\sum_{ij} a_{ij}x^iy^j$
 be the universal FGL. Then one computes the associated ``squared'' $2$-formal group law:
\begin{align*}
&\big(1+(x+y+\sum_{ij} a_{ij}x^iy^j)t\big)^2
 =1+2(x+y+a_{11}xy+...)t \\
 &\qquad\qquad+\lbrack x^2+2xy+y^2+2a_{11}(x^2y+xy^2)+(a_{11}^2+4a_{12})x^2y^2+\hdots\rbrack t^2.
\end{align*}

Then the coefficient of $x^iy^jt^l$ in this expression gives the image of the element $a_{ij}^l \in \Bus$
 under $\mathfrak S$. When $l=1$, this rule is simple: $\mathfrak S(a_{ij}^1)=2a_{ij}$.
 When $l=2$, one gets relations like:
\begin{align*}
&\mathfrak S(a_{1i}^2)=2a_{1,i-1}, \\
&\mathfrak S(a_{22}^2)=a_{11}^2+4a_{21}, \ \mathfrak S(a_{23}^2)=2a_{11}a_{12} + 2a_{13} + 2a_{22}, \\
&\mathfrak S(a_{22}^2)=2 a_{11} a_{22} + 2 a_{12}^{2} + 4 a_{23}.
\end{align*}
\end{rem}

\begin{num}
We now start with preparations which will lead to another more subtle functor $\FGL \to \tFGL$ which produces 2-FGL of the first type and is more closely related to characteristic classes.
 
Given a formal group law $F(x,y)$, we will denote by
$$
\bar x=-x+a_{11}x^2+(a_{12}-a_{11}-a_{11}^2)x^3+\hdots
$$
the unique power series in $x$ such that $F(x,\bar x)=0$ ---
 \emph{i.e.} the formal inverse, see \cite[Lem. 2.7]{Strickl}. For any composable $f\in R[[x]]$ we denote more generally by $\overline f$ the series 
 $$
\bar f:=\bar x(f)=-f+a_{11}f^2+(a_{12}-a_{11}-a_{11}^2)f^3+\hdots
$$

\begin{lm}
For any composable power series $f\in R[[x]]$, the series $\overline f$ is the unique series such that 
\[
F(f,\overline f)=0.
\]
\end{lm}

\begin{proof}
As $F(x,\bar x)=0$, it follows immediately, replacing $x$ by $f$ that $F(f,\overline f)=0$. Unicity is obvious: If $g\in R[[x]]$ is such that $F(f,g)=0$, we have 
\[
g=F(g,0)=F(g,F(f,\overline f))=F(F(g,f),\overline f)=\bar f.
\]
\end{proof}

For further use, it is worth noting that $\bar x$ induces an automorphism of the formal group law $F$, in the sense that 
\[
\overline{F(x,y)}:=\bar x(F(x,y))=F(\bar x,\bar y).
\]
The following lemma is probably well-known, but we include it for the sake of completeness.

\begin{lem}\label{lem:inverseseries}
Let $F\in R[[x,y]]$ be a (commutative) formal group law. Then, there exists a unique power series $G\in R[[x]]$ such that $x+\bar x=G(x\bar x)$.
\end{lem}

\begin{proof}
We start with a few preliminary observations. We have a ring homomorphism 
\[
R[[x,y]]\to R[[x]]
\]
such that $x\mapsto x$ and $y\mapsto \bar x$. As $\bar x=-x+\sum_{i\geq 2}a_ix^i$, we get $x+\bar x=x\bar x=0\pmod {x^2}$ and then the ideal $\langle x+y,xy\rangle^n$ is mapped under this homomorphism into $\langle x^{2n}\rangle$. The power series $F$ being symmetric, we may write it as 
\[
F(x,y)=x+y+uxy+\sum_{i+j\geq 2}a_{ij}(x+y)^i(xy)^j.
\]
We now construct for any $n\in\NN$ polynomials $G_n\in R[t]$ of degree $n$ by induction. We set $G_0=0, G_1=-ut$ and 
\[
G_n=-ut-\sum_{i+j=2}^na_{ij}G_{n-1}(t)^it^j \pmod{t^{n+1}}
\]
for $n\geq 2$. We now check by induction that for any $n\in\NN$ the following conditions are satisfied:
\begin{enumerate}
\item $x+\bar x=G_n(x\bar x)$ modulo $\langle x^{2n+2}\rangle$.
\item $G_{n}=G_{n-1} \pmod {t^{n}}$. 
\end{enumerate}

For $n=0$, we already know that $x+\bar x=0\pmod {x^2}$ and the second relation is empty. For $n=1$, we observe that 
\[
F(x,y)=x+y+uxy\pmod {\langle x+y,xy\rangle^2}
\]
and it follows that 
\[
0=F(x,\bar x)=x+\bar x+ux\bar x\pmod{x^4}.
\]
As $G_1=-ut$, both properties are satisfied. Let then $n\geq 2$ be such that properties 1. and 2. are satisfied for $n-1$.
Modulo $ \langle x+y,xy\rangle^{n+1}$, we have 
\[
F(x,y)=x+y+uxy+\sum_{i+j= 2}^{n}a_{ij}(x+y)^i(xy)^j
\]
and consequently
\[
0=F(x,\bar x)=x+\bar x+ux\bar x+\sum_{i+j= 2}^{n}a_{ij}(x+\bar x)^i(x\bar x)^j\pmod {x^{2n+2}}.
\]
By construction $x+\bar x=G_{n-1}(x\bar x)$ modulo $\langle x^{2n}\rangle$, i.e. $x+\bar x=G_{n-1}(x\bar x)+Hx^{2n}$ for some polynomial $H$. Modulo $\langle x^{2n+2}\rangle$, we have 
\[
(x+\bar x)^i=G_{n-1}(x\bar x)^i+iHG_{n-1}(x\bar x)^{i-1}x^{2n}
\]
and we can use the fact that $G_{n-1}(t)=0 \pmod t$ and $x\bar x=0$ modulo $\langle x^2\rangle$ to get $(x+\bar x)^i=G_{n-1}(x\bar x)^i\pmod {\langle x^{2n+2}\rangle}$. Altogether, the above expression reads as 
\[
0=x+\bar x+ux\bar x+\sum_{i+j= 2}^{n}a_{ij}G_{n-1}(x\bar x)^i(x\bar x)^j\pmod {x^{2n+2}}.
\]
This means that $x+\bar x=G_n(x\bar x)$ modulo $x^{2n+2}$, and the first point is satisfied. We are left with proving that $G_n=G_{n-1}$ modulo $t^n$. By definition, we may write 
\[
G_n=-ut-\sum_{i+j=2}^na_{ij}G_{n-1}(t)^it^j \pmod{t^{n+1}}.
\]
We know by induction that $G_{n-1}=G_{n-2}\pmod{t^{n-1}}$. It follows that $G_{n-1}t=G_{n-2}t$ modulo $t^n$. On the other hand, $G_{n-2}=0\pmod t$ and we obtain for $i\geq 2$ that 
\[
G_{n-1}^i=(G_{n-2}+\alpha t^{n-1})^i=G_{n-2}^i.
\]
Modulo $t^n$, we then obtain
\[
G_n=-ut-\sum_{i+j=2}^na_{ij}G_{n-1}(t)^it^j=-ut-\sum_{i+j=2}^na_{ij}G_{n-2}(t)^it^j=G_{n-1}.
\]
\end{proof}

\begin{rem}
Obviously, the conclusion of the lemma also holds for $-x\bar x$, i.e. there is a unique power series $G'$ such that $x+\bar x=G'(-x\bar x)$. We opted for the previous formulation in order to avoid sign issues.
\end{rem}

\begin{rem}\label{rem:finiteextension}
As a consequence of the above lemma, we see that the ring extension $R[[-x\bar x]]\subset R[[x]]$ is finite. Indeed, consider the polynomial
\[
(T-x)(T-\bar x)=T^2-(x+\bar x)T+x\bar x=T^2-G(x\bar x)T+x\bar x
\]
The extension  $R[[-x\bar x]]\subset R[[-x\bar x]][T]/(T^2-G(x\bar x)T+x\bar x)$ is finite and there is an isomorphism $R[[-x\bar x]][T]/(T^2-G(x\bar x)T+x\bar x)\to R[[x]]$ mapping $T$ to $x$.
\end{rem}

\end{num}
\begin{lm}\label{lem:unique22}
Let $F$ be a formal group law.
 Then there exists a unique $(2,2)$-series $\bar F_t^2(x,y)$
 such that the following equality holds in $R[[x,y]][t]$:
$$
\bar F^2_t(-x\bar x,-y\bar y)=(1-F(x,y)F(\bar x,\bar y)t)(1-F(\bar x,y)F(x,\bar y)t).
$$
The series $\bar F_t^2(x,y)$ is of type I.
\end{lm}

\begin{proof}
As the series $-x\bar x$ is of the form $x^2 \mod (x^3)$,
 the endomorphism
$$
\phi_F:R[[x,y]] \rightarrow R[[x,y]], x \mapsto -x\bar x, y \mapsto -y\bar y
$$
is injective. This proves the uniqueness of $\bar F_t^2(x,y)$.
For the existence, one has to prove that the series:
$$
F(x,y)F(\bar x,\bar y)F(\bar x,y)F(x,\bar y)
 \text{ and }
F(x,y)F(\bar x,\bar y)+F(\bar x,y)F(x,\bar y)
$$
belongs to the image of $\phi_F$. This is precisely \cite[Lemma 2.21]{BN}.

We now prove that $\bar F_t^2(x,y)$ is of type I. For this, it suffices to prove that $F_2(x,x)=0$. Since, $F(\bar x,x)=0$, we obtain
\begin{eqnarray*}
\bar F^2_t(-x\bar x,-x\bar x) & = & (1-F(x,x)F(\bar x,\bar x).t)(1-F(\bar x,x)F(x,\bar x).t) \\
& = & (1-F(x,x)F(\bar x,\bar x).t).
\end{eqnarray*}
\end{proof}

\begin{rem}\label{rem:rootsN}
In view of Remark \ref{rem:finiteextension}, we see that the roots of $\bar F^2_t$ can be chosen to be $F^{[1]}=-F(x,y)F(\bar x,\bar y)$ and $F^{[2]}=-F(\bar x,y)F(x,\bar y)$ in $R[[x,y]]$.
\end{rem}

In the following statement involving two formal group laws $F$ and $G$, we denote by $\overline{x}^F$ the formal inverse of $x$ with respect to the law $F$ and $\overline{x}^G$ the one with respect to $G$. Also, given a composable series $\Theta$ we denote by 
\[
\ev_\Theta:R[[x]]\to R[[x]]
\]
the homomorphism given by $f(x)\mapsto f(\Theta)$. The following results are mild variations on \cite[\S 3, Lemma 3.1]{Bu78}.
 
 \begin{lm}\label{lem:morphisms22}
  Let $F(x,y)$ and $G(x,y)$ be formal group laws, and let $\Theta:F\to G$ be a morphism of formal group laws. Then, there is a unique composable series $\varphi(\Theta)$ such that the following diagram commutes
  \[
 \xymatrix{R[[x]]\ar[r]^-{\ev_{\varphi(\Theta)}}\ar[d]_-{\ev_{-x\bar x^G}} & R[[x]]\ar[d]^-{\ev_{-x\bar x^F}} \\
 R[[x]]\ar[r]_-{\ev_\Theta} & R[[x]].}
 \]
 Further, given a morphism of formal group laws $\Psi:G\to H$, we have 
 \[
 \varphi(\Theta\circ \Psi)=\varphi(\Theta)\circ\varphi(\Psi).
 \]
 \end{lm}
 
 \begin{proof}
 As $\Theta$ is composable, we have $\Theta(F(x,\overline{x}^F))=0$ and then 
 \[
 G(\Theta,{\Theta}(\overline{x}^F))=0
 \]
 i.e. ${\Theta}(\overline{x}^F)=\overline{\Theta}^G={\bar x}^G(\Theta)=\ev_{\Theta}(\bar x^G)$. It follows that 
 \[
-(x\bar x^G)(\Theta)=-\Theta\bar x^G(\Theta)=-\Theta\Theta(\bar x^F).
 \]
 This power series is obviously symmetric in the variables $x$ and $\bar x^F$, and we may write it in terms of the elementary symmetric polynomials in $x$ and $\bar x$. Lemma \ref{lem:inverseseries} then gives a series $\varphi(\Theta)$ fulfilling the stated property. The unicity statement is obvious from the fact that $x\mapsto -x\bar x^F$ is injective. Finally, we have $\varphi(\Theta)(-x\bar x^F)=-\Theta\Theta(\bar x^F)$, from which we immediately deduce that $\varphi(\Theta)$ is composable.
 
The statement about the composite follows from the commutative diagram
\[
\xymatrix{R[[x]]\ar[r]^-{\ev_{\varphi(\Psi)}}\ar[d]_-{\ev_{-x\bar x^H}} & R[[x]]\ar[r]^-{\ev_{\varphi(\Theta)}}\ar[d]_-{\ev_{-x\bar x^G}} & R[[x]]\ar[d]^-{\ev_{-x\bar x^F}} \\
R[[x]]\ar[r]_-{\ev_\Psi} & R[[x]]\ar[r]_-{\ev_{\Theta}} & R[[x]]}
\]
and the unicity statement of the previous point.
 \end{proof}

\begin{prop}\label{prop:functorN}
Let $R$ be a ring and consider the notation of the previous lemma. Then the assignment
$$
N_R:\FGL(R) \rightarrow(\tFGL)_I(R), F(x,y) \mapsto \bar F^2_t(x,y)
$$
can be extended into a functor mapping a morphism $\Theta$ in $\FGL(R)$ to $\varphi(\Theta)$. 
\end{prop}
 
\begin{proof}
We already know from Lemma \ref{lem:unique22} that the functor is well-defined on objects. We now consider the commutative diagram
\[
\xymatrix{ R[[x,y]]\ar[r]^-{\ev_{\varphi(\Theta)}}\ar[d]_-{\phi_G} & R[[x,y]]\ar[d]^-{\phi_F} \\
 R[[x,y]]\ar[r]_-{\ev_{\Theta}} & R[[x,y]]}
\]
where $\phi_G(f(x,y))=f(-x\bar x^G,-y\bar y^G)$ (idem for $F$) and
 $$\ev_{\varphi(\Theta)}(f(x,y))=f(\varphi(\Theta)(x),\varphi(\Theta)(y))$$
 (idem for $\ev_\Theta$).
 Next, we consider the polynomial
\[
1+\bar G_1^2(\varphi(\Theta)(x),\varphi(\Theta)(y))t+\bar G_2^2(\varphi(\Theta)(x),\varphi(\Theta)(y))t^2.
\]
Applying to $-x\bar x^F$, using $\varphi(\Theta)(-x\bar x^F)=-\Theta\Theta(\bar x^F)$ and $\Theta(\bar x^F)=\overline{\Theta(x)}^G$ we obtain
\[
(1+G(\Theta(x),\Theta(y))G(\Theta(\bar x^F),\Theta(\bar y^F))t)(1+G(\Theta(\bar x^F),\Theta(y))G(\Theta(x),\Theta(\bar y^F))t).
\]
As already observed in the proof of Lemma \ref{lem:morphisms22}, we have $\bar x^G(\Theta)=\Theta(\bar x^F)$ and we then see that the previous polynomial is of the form
\[
(1+\Theta(F(x,y))\Theta(F(\bar x^F,\bar y^F))t)(1+\Theta(F(\bar x^F,y))\Theta(F(x,\bar y^F))t).
\]
As $F(\bar x^F,\bar y^F)=\overline{F(x,y)}^F$ and $F(x,\bar y^F)=\overline{F(\bar x^F,y)}^F$, we finally obtain a polynomial of the form
\[
(1+\Theta(F(x,y))\Theta(\overline{F(x,y)}^F)t)(1+\Theta(F(\bar x^F,y))\Theta(\overline{F(\bar x^F,y)}^F)t).
\]
As $\Theta\Theta(\bar x^F)=\varphi(\Theta)(x\bar x^F)$, the above expression becomes 
\[
(1+\varphi(\Theta)(F(x,y)\overline{F(x,y)}^F)t)(1+\varphi(\Theta)(F(\bar x^F,y)\overline{F(\bar x^F,y)}^F)t).
\]
which is $\varphi(\Theta)(\bar F^2_t(-x\bar x^F,-y\bar y^F))$. We have then obtained that 
\[
\bar G_t^2(-x\bar x^F,-y\bar y^F)=\varphi(\Theta)(\bar F^2_t(-x\bar x^F,-y\bar y^F))
\]
and we deduce from the injectivity of the map $x\mapsto -x\bar x^F$ that $\varphi(\Theta)$ is a morphism between $\bar F^2_t$ and $\bar G^2_t$ (see Definition \ref{df:morphism_nd-series}). The functor respects compositions by Lemma \ref{lem:morphisms22} and it is straightforward to check that $\varphi(\mathrm{Id})=\mathrm{Id}$. 
 \end{proof}
 
 \begin{rem}\label{rem:2fgladd}
 The functor $N_R$ is not faithful. Consider the additive formal group law $F(x,y)=x+y$ and the morphism $\Theta:F\to F$ given by the series $\Theta(t)=-t$ (i.e. $\Theta=\bar x$). Then, we see that the series $\Theta'(t)=t$ and $\Theta$ have the same image under the functor. Indeed, 
\[
-\Theta(x)\Theta(\bar x)=--x(-\bar x)=x^2,
\]
showing that $\varphi(\Theta)=t$. The same applies to $\Theta'$.

More generally, let $F(x,y)$ be an arbitrary formal group law. Then, we know that the series $\bar x^F=\bar x$ is an endomorphism of $F$. As $\bar x(\bar x)=x$, it follows  that $\varphi(\bar t)=t=\varphi(t)$. 
 \end{rem}
 
 \begin{rem}
 The functor $N$ is not full either. Consider the elementary $2$-formal group of the first type
 \[
 F_t(x,y)=1+2(x+y)+(x-y)^2t^2.
 \]
 A straightforward computation shows that it is the image of the additive formal group law under the functor $N$. Now, consider the series $\varphi(x)=-x$. It is an endomorphism of $F_t$: If $F^{[1]}$ and $F^{[2]}$ are the roots of $F$, we have $F^{[1]}+F^{[2]}=2(x+y)$ and $F^{[1]}F^{[2]}=(x-y)^2$. Now 
\[
(1-F^{[1]}t)(1-F^{[2]}t)=1-2(x+y)t+(x-y)^2t^2
\]
which is apparently $F_t(-x,-y)$. If $-1$ is not a square in $R$, then $\varphi$ is not the image of an endomorphism of the additive formal group law. Indeed, if $\Theta$ is such an endomorphism mapping to $\varphi$, then we should have 
\[
-x^2=\varphi(-x\bar x)=-\Theta\Theta(\bar x)
\]
Writing $\Theta=ax+\sum_{i\geq 2}a_ix^i$, we see that the above equality implies $a^2=-1$.
 \end{rem}
 
These two observations might suggest to restrict the functor to the subcategories where morphisms are \emph{strict} isomorphisms. One might ask if the functor then becomes an equivalence, at least after inverting $2$.

%
%

\section{Formal ternary laws}\label{sec:FTL}

\subsection{Algebraic definitions}\label{sec:df_FTL}

\begin{num}
Our base ring will be the ring
$$
\ZZe:=\GW(\ZZ),
$$
\emph{i.e.} the Grothendieck-Witt ring associated with the category of finitely generated projective $\ZZ$-modules equipped with a non-degenerate symmetric bilinear form.
 Let $<-1>$ be the class of the bilinear form 
 \[
 \ZZ\times \ZZ\to \ZZ
 \] 
 given by $(x,y)\mapsto -xy$.
 Following Morel's conventions in motivic homotopy theory, we set $\epsilon:=-<-1>$.
 Then the canonical map:
$$
\ZZ[t] \rightarrow \GW(\ZZ), t \mapsto \epsilon
$$
is an epimorphism whose kernel is the ideal $(t^2-1)$
 (see \cite[chap. 2]{MilHus}). One writes suggestively $\ZZe=\ZZ[\epsilon]/(\epsilon^2-1)$.
 In particular, $\ZZe$ is a free $\ZZ$-module of rank $2$:
 any element can uniquely be written as $a+b.\epsilon$ for a couple of integers $(a,b)$.

One defines the plus-part (resp. minus-part) of $\ZZe$ as the quotient ring:
$$
\ZZe^+=\ZZe/(\epsilon+1) \text{ resp. } \ZZe^-=\ZZe/(\epsilon-1).
$$
We have an obvious identification $\ZZe^+ \simeq \ZZ$ (resp. $\ZZe^- \simeq \ZZ$)
 obtained by mapping the class of $a+b\epsilon$ to $a-b$ (resp. $a+b$).
 The following lemma is now obvious.
\end{num}
\begin{lm}
The canonical morphism of rings
$$
\ZZe \rightarrow \ZZe^+ \times \ZZe^- \simeq \ZZ\times \ZZ
$$
is injective and its image is given by the set of pairs $(n,m)$
 such that $n \equiv m \mod 2$. It is an isomorphism after inverting $2$
 on both sides.

Geometrically, $\Spec(\ZZe)$ is the union of two closed subschemes whose reductions are
 $\Spec(\ZZ)$, and which intersects in a single point, $\Spec(\mathbb F_2[t]/(t+1)^2)$.
\end{lm}

\begin{cor}
For any $\ZZe$-algebra $A$, one puts $A^\pm=A \otimes_{\ZZe} \ZZe^\pm$. \\
 There is a canonical map: $A \rightarrow A^+ \times A^-$ which is a monomorphism
 if $A$ is flat over $\ZZe$, and an isomorphism if $2$ is invertible in $A$.
\end{cor}

\begin{ex}
The motivation underlying our choice of base ring comes from the following example.
 Let $S$ be any scheme, and $\E$ a ring spectrum over $S$.
 Then the coefficient ring $\E^{**}(S)$ is canonically a $\ZZe$-algebra,
 where the action of $\epsilon$ is given by the canonical map
 $\epsilon:\GGx S \rightarrow \GGx S$ seen in $\SH(S)$ after infinite suspension,
 induced by the inverse of the group $S$-scheme $\GGx S$.
\end{ex}

The following definition is a variation of the notion of a $(n,d)$-group
 (Def. \ref{df:nd_groups}) which naturally arise in the theory of characteristic classes
 related with $\Sp$-oriented theories (see the next section). It has been partially introduced in \cite[section 2.3]{DF21}, with origins in unpublished notes of Ch. Walter.
 
\begin{df}\label{df:FTL}
Let $R$ be a $\ZZe$-algebra.
A formal ternary law, FTL for short,
 with coefficients in $R$ is a $(4,3)$-series with coefficients in $R$
$$
F_t(\ux)=1+F_1(x,y,z)t+F_2(x,y,z)t^2+F_3(x,y,z)t^3+F_4(x,y,z)t^4
$$
satisfying the following properties:
\begin{enumerate}
\item \textit{Neutral element}. The $(4,1)$-series $F_t(x,0,0)$ is split
with roots $x$ and $-\epsilon.x$ each with multiplicity $2$,
 \emph{i.e.} $F_t(x,0,0)=(1+xt)^2(1-\epsilon xt)^2$.
\item \textit{Symmetry}. The element $F_t(x,y,z)$ of $R[[x,y,z]][t]$ is a fixed
point under the action of the group $\mathfrak S(x,y,z)$ permuting the formal variables.
\item \textit{Associativity}. Given formal variables $(x,y,z,u,v)$,
 one has the following equality of $(16,5)$-series:
\begin{equation}
\begin{split}
&F_t\big(F_t(x,y,z),u,v\big) \\
&\quad =F_t\big(x,F_t(y,z,u),v\big)
\end{split}
\end{equation}
where we have used the substitution operation (Definition \ref{df:substitution}),
valid as $F_t(x,y,z)$ is substituable by property (1).
\item \textit{$\epsilon$-Linearity}. One has the following relation in $\MForm_{4,3}(R)$:
$$
F_t(-\epsilon x,y,z)=F_{-\epsilon t}(x,y,z).
$$
\item \textit{Weak neutral element}. The following relation holds:
$$
F_4(x,x,0)=0.
$$
\end{enumerate}


We let $\FTL$ be the subcategory of $\MForm_{4,3}^\circ$ whose objects are as above,
 and whose morphisms $(\varphi,\Theta):(R,F_t) \rightarrow (R',F'_t)$ are given
 by morphisms of $(4,3)$-series such that $\varphi:R \rightarrow R'$ is $\ZZe$-linear.
\end{df}

\begin{rem}\label{rem:FTL_basics}
\begin{enumerate}
\item The symmetry and associativity axioms imply two other associativity axioms:
\begin{align*}
\begin{split}
F_t\big(F_t(x,y,z),u,v\big)=F_t\big(x,y,F_t(z,u,v)\big), \\
F_t\big(x,F_t(y,z,u),v\big)=F_t\big(x,y,F_t(z,u,v)\big).
\end{split}
\end{align*}
\item All FTL arising from motivic symplectic oriented spectra (see section \ref{sec:geometry} below) we have studied so far satisfy a stronger version of the weak neutral element axiom, namely we have $F_3(x,x,0)=0$. One may wonder if this stronger condition holds for all geometric examples and should be added as an axiom. 
\item Beware that a FTL is not a $(4,3)$-group in general because of the "Neutral element" axiom.
 However, if $\epsilon$ acts as $-1$ on the coefficient ring $R$, a FTL is precisely a $(4,3)$-group
 which satisfy the additional ``Weak neutral element" axiom.
\end{enumerate}
\end{rem}

\begin{num} \label{num:FTL_explicit_axioms}
Writing 
\begin{equation}\label{eq:generitc_FTL}
F_t(x,y,z)=1+\sum_{i,j,k\geq 0, 1 \leq l \leq 4} a_{ijk}^lx^iy^jz^kt^l,
\end{equation}
 one can restate the above properties in term of the coefficients $a_{ijk}^l$ as follows:
\begin{enumerate}
\item[(1)] \textit{Symmetry}. for all $l$, $a_{ijk}^l$ is invariant under permutations of the triple $(i,j,k)$.
\item[(2)] \textit{Neutral element}.
 $a_{i00}^l=\begin{cases}
1 & i=l=4, \\
2(1-\epsilon) & i=l=1,3, \\
2(1-2\epsilon) & i=l=2, \\
0 & \text{otherwise.}
\end{cases}$
\item[(4)] \textit{$\epsilon$-Linearity}. The relation $(1+\epsilon)a_{ijk}^l$ holds
 whenever $l$ and one of the $i,j,k$ don't have the same parity.
\item[(5)] \textit{Weak neutral element}. For all $n \geq 0$, $\sum_{i+j=n} a^4_{ij0}=0$.
\end{enumerate}
\end{num}

\begin{rem}
\begin{enumerate}
\item The $\epsilon$-linearity axiom exhibits $\ZZe$-torsion amongst the coefficients on an FTL.
 When projecting to the plus-part, this torsion property disappears, and when projecting to the 
 minus-part, it becomes a $2$-torsion assertion.
\item The associativity formula which we omitted in the list is indeed very difficult to compute, and even to print out! We have worked it out
 partially using computers and we refer the interested reader to the appendix.
 Of course, this amounts to write explicitly the relations for 
\[
\prod_{i=1}^4F_t(F^{[i]},u,v)=\prod_{i=1}^4F_t(x,F^{[i]},v)
\]
For instance the relations in degree $1$ are given by $\sum_{i=1}^4F_1(F^{[i]},u,v)=\sum_{i=1}^4F_1(x,F^{[i]},v)$.
\end{enumerate}
\end{rem}

The following result is an obvious extension of \Cref{prop:universal_nd-group}.
\begin{prop}\label{prop:universal_FTL}
The category $\FTL$ admits an initial object, $(\Walt,F_t^{univ})$.

If one writes $a_{ijk}^l$ the coefficients of the universal FTL $F_t^{univ}(x,y,z)$,
 the $\ZZe$-algebra $\Walt$ is generated by the countable family
 $\{a_{ijk}^l, i\geq j\geq k\geq 0, 1 \leq l, \leq 4\}$.
 It is a quotient ring of the form
$$
\Walt=\ZZe\big[a_{ijk}^l\big]/\mathcal R_{FTL}
$$
and the ideal of relations $\mathcal R_{FTL}$ is generated by homogeneous elements of the polynomial ring
 $\ZZe\big[a_{ijk}^l\big]$ with respect to the degree defined by 
$$
\deg(a_{ijk}^l)=i+j+k-l, (\text{and $\deg(\lambda)=0$ for $\lambda \in \ZZe$} ).
$$
\end{prop}
The proof works as in \emph{loc. cit.} Note in particular that the supplementary axioms
 of $\epsilon$-linearity and weak neutral element readily produce homogeneous relations
 as explained in \Cref{num:FTL_explicit_axioms}.

\begin{df}
We call the $\ZZ$-graded ring $\Walt$ defined above the Walter ring.
\end{df}

\begin{defn}
We say that an object $(R,F)$ in $\FTL$ is \emph{oriented} provided we have $\epsilon=-1\in R$.
\end{defn}

The term ``oriented'' is motivated by the fact that oriented spectra lie in the $+$-part of $\SH(k)$ where $\epsilon=-1$. Our next aim is now to determine the universal FTL of degree $\leq 0$, i.e. the ring $\Walt^{\leq 0}$ with associated law $F^{\leq 0}_t(x,y,z)$. The nonzero coefficients are of the form 
\[
\{a^i_{jkl}\vert 1\leq i\leq 4, 0\leq j+k+l\leq i\}.
\] 
The algorithm described in the appendix allows to find the result with integral coefficients. There are however 20 generators and 54 relations and the result would be quite long to state.  For this reason, we prefer to state the result with $\ZZ[\frac 12]$-coefficients. 

\begin{thm}\label{thm:compute_universal_FTL}
Let $\Walt^{\leq 0}[\frac 12]$ be the sub-ring of $\Walt[\frac 12]$ generated by variables of non-positive degree. Then, one has
\[
\Walt^{\leq 0}[\frac 12]\simeq \ZZe[\frac 12][a]/\langle (a-40)(1-\epsilon),(a-8)(a+8)(1+\epsilon)\rangle
\]
and the coefficients of the universal FTL are given by 
\[
\begin{array}{llll}
a^1_{100}=2(1-\epsilon) & & & \\
a^2_{200}=2(1-2\epsilon) & a^2_{110}=2(1-\epsilon) & & \\
a^3_{300}=2(1-\epsilon) & a^3_{210}=-2(1-\epsilon) & a^3_{111}=a & \\
a^4_{400}=1 & a^4_{310}=-2(1-\epsilon) & a^4_{220}=2(1-2\epsilon) & a^4_{211}=2(1-\epsilon).
\end{array}
\]
\end{thm}

\begin{proof}
According to Appendix \ref{appA:sample}, the parameters have all the stated values, except for $a:=a^3_{111}$ which should satisfy two relations, namely $a^2=832-768\epsilon$ and $(a-40)(\epsilon-1)=0$. In $\ZZe[\frac 12]$, we have the idempotents $e:=\frac 12(1-\epsilon)$ and $f:=\frac 12(1+\epsilon)$. One checks readily that 
\[
a^2-832+768\epsilon=((a-40)e+(a-8)f)((a-40)e+(a+8)f)+80(a-40)e
\]
so that the relevant ideal is $\langle (a-40)e,(a-8)(a+8)f\rangle$.
\end{proof}

\begin{rem}\label{rem:compute_universal_FTL}
The theorem states that there are two laws of degree $\leq 0$ (when $2$ is inverted). Indeed, one has either $a^3_{111}=8(2-3\epsilon)$ or $a^3_{111}=8(3-2\epsilon)$. Setting $F_t(x,y,z)$ for the first law (i.e. with coefficients as above and $a^3_{111}=8(2-3\epsilon)$) and $G_t(x,y,z)$ for the second one, we see that 
\[
F_t(-\epsilon x,-\epsilon y,-\epsilon z)=G_t(x,y,z).
\]
Comparing with \cite[Definition 2.1.3, Theorem 3.3.3]{DF21}, see also the next example, only one of these two laws
is known to arise from an Sp-oriented spectrum and then satisfies the normalization axiom. We may then define a new Borel class $b_1^\prime$ as the class $-\epsilon b_1$ and keep $t$ fixed. The law $G_t(x,y,z)$ is obtained out of $F_t(x,y,z)$ in this fashion, but $F_t$ and $G_t$ are not isomorphic in the sense of FTL. 
\end{rem}

\subsection{Examples}\label{examples:ftl}
Let $k$ be a field with $6 \in k^{\times}$ and consider the ring spectrum $\mathbf{H}_{\mathrm{MW},k}$ over $k$
 representing Milnor-Witt motivic cohomology, introduced in \cite[Chap. 6]{BCDFO21}. 
 The FTL $F_t(x,y,z)$ associated to this cohomology theory was computed in \cite[Theorem 3.3.2]{DF21}. We have:
\[
F_i=
\begin{cases}
2(1-\epsilon)\sigma(x), & i=1, \\
2(1-2\epsilon)\sigma(x^2)+2(1-\epsilon)\sigma(xy), & i=2, \\
2(1-\epsilon)\sigma(x^3)-2(1-\epsilon)\sigma(x^2y)+8(2-3\epsilon)xyz, & i=3, \\
\sigma(x^4)-2(1-\epsilon)\sigma(x^3y)+2(1-2\epsilon)\sigma(x^2y^2)+2(1-\epsilon)\sigma(x^2yz), & i=4.
\end{cases}
\]
We note that $F_t$ is both of degree and valuation $0$. If we invert $2$, we observe that this law is the universal one of degree $0$. In a sense, this is a justification a posteriori of the terminology of \cite[Definition 3.3.3]{DF21}.

Let now $\KO$ be the absolute ring spectrum representing Hermitian $K$-theory over $\ZZ[\frac 12]$. There exists an element $\gamma\in \KO^{8,4}(\ZZ[\frac 12])$ such that multiplication by $\gamma$ induces an isomorphism $\KO^{*,*}\to \KO^{*+8,*+4}$. We denote as usual its inverse by $\gamma^{-1}\in  \KO^{-8,-4}(\ZZ[\frac 12])$. On the other hand, we have $\KO^{4,2}(\ZZ[\frac 12])=\mathrm{GW}^-(\ZZ[\frac 12])$, the Grothendieck-Witt group of skew symmetric (non degenerate) bilinear forms. We have $\mathrm{GW}^-(\ZZ[\frac 12])=\ZZ\cdot \tau$, where $\tau$ is the class of the symplectic form $\ZZ[\frac 12]^2\times \ZZ[\frac 12]^2\to \ZZ[\frac 12]$ defined by $((u_1,v_1),(u_2,v_2))\mapsto u_1v_2-u_2v_1$. The following computation can be found in \cite[Theorem 6.6]{FH21}.

\begin{thm}\label{koftl}
The formal ternary laws $F_i=F_i(x,y,z)$ of Hermitian $K$-theory are
\small{\[
F_1=2(1-\epsilon)\sigma(x)+ \tau\gamma^{-1} \sigma(xy)+ \gamma^{-1}xyz, \hspace{10cm} 
\]
\[
F_2=2(1-2\epsilon)\sigma(x^2)+2(1-\epsilon)\sigma(xy)+2\tau\gamma^{-1}\sigma(x^2y)-3\tau\gamma^{-1} xyz+\gamma^{-1}\sigma(x^2y^2), \hspace{3cm} 
\]
\[
F_3= 2(1-\epsilon)\sigma(x^3)-2(1-\epsilon)\sigma(x^2y)+8(2-3\epsilon)xyz+\tau \gamma^{-1}\sigma(x^3y)-2\tau\gamma^{-1} \sigma(x^2y^2)+3\tau\gamma^{-1}\sigma(x^2yz)+ \gamma^{-1}\sigma(x^3yz),
\]
\[
F_4= \sigma(x^4)-2(1-\epsilon)\sigma(x^3y)+2(1-2\epsilon)\sigma(x^2y^2)+2(1-\epsilon)\sigma(x^2yz)-\tau\gamma^{-1}\sigma(x^3yz)+2\tau\gamma^{-1}\sigma(x^2y^2z)+\gamma^{-1}\sigma(x^2y^2z^2).
\]}
\end{thm}

The computation for $i=1$ is also in unpublished notes of Ch. Walter. As explained in \cite[Remark 6.7]{FH21}, setting $\epsilon=1$ and $\tau=0$ one recovers Ananyevskiy's computation \cite{An17} for Witt cohomology. Setting $\epsilon=-1$, $\tau=2a^{-2}$ and $\gamma^{-1}=a^4$, one recovers the image of the multiplicative formal group law $x+y-axy$ under the composite of functors $W=C \circ N$ (see Section \ref{Comparison}), which we have calculated independently by computer (note that the FGL $x+y+axy$ gives the same result).

\section{Comparison}\label{Comparison}

We constructed a functor $N: \FGL \rightarrow (\tFGL)_I$ in Proposition \ref{prop:functorN}. In this section, we will produce a functor 
$C: (\tFGL)_I \rightarrow \FTL$ and thus obtain a composite functor 
\[
W:=C\circ N: \FGL\to \FTL.
\]
The motivation of the construction of $C$ is of geometric nature (compare also section \ref{sec:geometry} below). Consider the universal bundles $U_1,U_2$ and $U_3$ on $(\mathrm{H}\Proj^{\infty})^{\times 3}=\mathrm{BSp}_2^{\times 3}$. We write as usual $x$ for the first Borel class of $U_1$, $y$ for the first Borel class of $U_2$ and $z$ for the first Borel class of $U_3$. Suppose that we can associate characteristic classes to $U_1U_2$ under the form of a reasonable series
\[
F_t(x,y)=1+F_1(x,y)t+F_2(x,y)t^2,
\]
for instance for $F_t$ a two-valued formal group law. The roots $F^{[1]}$ and $F^{[2]}$ represent characteristic classes of bundles $V_1$ and $V_2$ having the property that 
\[
U_1U_2=V_1\oplus V_2
\]
over a suitable extension of $(\mathrm{H}\Proj^{\infty})^{\times 2}$. We could then write
\[
U_1U_2U_3=(V_1\oplus V_2)\cdot U_3=V_1U_3\oplus V_2U_3
\]
The characteristic classes of the first bundle are given by the polynomial
\[
1+F_1(F^{[1]},z)t+F_2(F^{[1]},z)t^2
\]
and the ones of the second bundle by the polynomial
\[
1+F_1(F^{[2]},z)t+F_2(F^{[2]},z)t^2
\]
The product of the two polynomials is the well-defined series $C_t(F):=F_t(F_t(x,y),z)$. Explicitly, we have
{\small \[
C_i(F)(x,y,z)=\begin{cases} F_1(F^{[1]},z)+F_1(F^{[2]},z) & \text{ if $i=1$,} \\
F_2(F^{[1]},z)+F_2(F^{[2]},z)+F_1(F^{[1]},z)F_1(F^{[2]},z) & \text{ if $i=2$,}  \\
  F_1(F^{[1]},z)F_2(F^{[2]},z)+ F_2(F^{[1]},z)F_1(F^{[2]},z) & \text{ if $i=3$,}  \\
F_2(F^{[1]},z)F_2(F^{[2]},z) & \text{ if $i=4$.}  
\end{cases}
\]}

\begin{lem}\label{lem:2FGLto4FTL}
Let $F_t$ be a 2-valued formal group law of the first type over a ring $R$. Then, the series $C_t(F)(x,y,z):=F_t(F_t(x,y),z)$ is a FTL over the ring $R$. 
\end{lem}

\begin{proof}
We start with the symmetry axiom. As $F_t(x,y)=F_t(y,x)$, we obviously have $C_t(F)(x,y,z)=C_t(F)(y,x,z)$. To prove that $C_t(F)(x,y,z)=C_t(F)(x,z,y)$, we use the associativity and commutativity axioms for $F$
\[
F_t(F_t(x,y),z)= F_t(x,F_t(y,z))=F_t(x,F_t(z,y))=F_t(F_t(x,z),y).
\]
For the neutral element, we note that $\epsilon$ acts as $-1$ on $R$, which implies that we have to prove that
\[
C_t(F)(x,0,0)=1+4xt+6x^2t^2+4x^3t^3+x^4t^4.
\]
This is a direct computation using $F_t(x,0)=1+2xt+x^2t^2$. For the weak neutral element axiom, we observe that $F_2(x,x)=0$ implies that $F_t(x,x)=1+F_1(x,x)t$ so that
\[
C_t(F)(x,x,0)=F_t(F_t(x,x),0)=F_t(F_1(x,x),0)=(1+F_1(x,x)t)^2
\] 
yielding obviously $C_3(F)(x,x,0)=C_4(F)(x,x,0)=0$.
Finally, we have
\begin{eqnarray*}
C_t(F)(C_t(F)(x,y,z),v,w) & = & C_t(F)(F_t(F_t(x,y),z),v,w) \\
 & = & F_t(F_t(F_t(F_t(x,y),z),v),w) \\
 & = & F_t(F_t(F_t(x,F_t(y,z)),v),w) \\ 
  & = & F_t(F_t(x,F_t(F_t(y,z),v)),w) \\
  & = & C_t(F)(x,C_t(F)(y,z,v),w).
\end{eqnarray*}
\end{proof}

\begin{ex}
Consider the two-valued formal group law 
\[
F_t(x,y)=1+2(x+y)t+(x-y)^2t^2.
\]
which by Remark \ref{rem:2fgladd} is the image of the additive formal group law (associated to Chow groups) under $N$.
Write as usual $F^+$ and $F^-$ for the roots of $F_t$. We have 
\[
F_t(F_t(x,y),z)  =  F_t(F^+,z)F_t(F^-,z) 
\]
and a direct computation yields
\[
C(F)_i=\begin{cases} 4\sigma(x) & \text{ if $i=1$.} \\
6\sigma(x^2)+4\sigma(xy) & \text{ if $i=2$.} \\
4\sigma(x^3)-4\sigma(x^2y)+40xyz & \text{ if $i=3$.} \\
\sigma(x^4)-4\sigma(x^3y)+6\sigma(x^2y^2)+4\sigma(x^2yz)& \text{ if $i=4$.}
\end{cases}
\]
This is the FTL of Chow groups. Note that setting $\epsilon=-1$ in $\HMW$ (the first example of subsection \ref{examples:ftl}) yields the same result, as expected.
\end{ex}

If we work more generally with unoriented 2-valued formal group laws over a $\ZZe$-algebra as in Remark \ref{rem:noepsilon}, applying $C$ still is associative and symmetric, arguing as in the previous lemma. Also, $\epsilon$-linearity of $F_t$ implies $\epsilon$-linearity of $C(F_t)$, which in turn implies the $\epsilon$-refinement on the axiom about the neutral element arguing as in \cite[Proposition 2.3.4(1)]{DF21}. However, such laws might not satisfy the weak neutral element axiom, as the next example shows.

\begin{prop}\label{thm:HMWisnot2fgl}
Consider the $\ZZe$-variant $F^{univ,0}_{\epsilon}$ over $\Bus_{\epsilon}^0$ of the degree zero universal two-valued formal group law as studied after Proposition \ref{prop:Bus_gradings}. 
Then $C(\Bus_{\epsilon}^0,F^{univ,0}_{\epsilon})$ is given by  
{\small \[
C(\theta_{ue})_i=
\begin{cases}
2(1-\epsilon)\sigma(x), & i=1, \\
2(1-2\epsilon)\sigma(x^2)+(a+4)(1-\epsilon)\sigma(xy), & i=2, \\
2(1-\epsilon)\sigma(x^3)+2(1-\epsilon)(1+a)\sigma(x^2y)+2(1-\epsilon)(a^2-a+4)\sigma(xyz), & i=3, \\
\sigma(x^4)+2a\sigma(x^3y)+(2+a^2)\sigma(x^2y^2)+2(a^2+a)\sigma(x^2yz), & i=4.
\end{cases}
\]}
\end{prop}
\begin{proof}
Straightforward computation. Use the relation $(1 + \epsilon)a=0$ for $i=2,3,4$ as well as the cubical relation for $a$ in $F_4$. 
\end{proof}

Looking at the oriented part, i.e. setting $\epsilon=-1$, we recover the previous oriented example. Setting $a=-(1-\epsilon)$, we obtain an unoriented formal ternary law which is slightly different from the one for MW-motivic cohomology computed by the first two authors in \cite{DF21} and recalled in Example \ref {examples:ftl} above. Namely, the coefficient of $\sigma(xyz)$ in $F_3$ specializes to $20(1-\epsilon)$ instead of $8(2-3\epsilon)$, and the one of $\sigma(x^2y^2)$ in $F_4$ to $2(2-\epsilon)$ instead of $2(1-2\epsilon)$. So they differ by multiples of $1 + \epsilon$. In particular, the law of Proposition \ref{thm:HMWisnot2fgl} does not specialise to the formal ternary law for $\HMW$. Moreover, we have $F_4(x,x,0)=2(1 + \epsilon)x^4$, hence the weak neutral element axiom does not hold!

\begin{prop}
Let $R$ be a ring or more generally a $\ZZe$-algebra, and consider the notations of the previous lemma. Then the assignment
$$
C_R:(\tFGL)_{I}(R) \rightarrow \FTL, F(x,y) \mapsto C_t(F)(x,y,z)
$$
extends into a functor which is the identity on morphisms. 
\end{prop}

\begin{proof}
Let $F$ and $G$ be two valued formal groups and let $\Theta$ be a composable power series such that 
\[
\Theta_t(F_t(x,y))=G_t(\Theta(x),\Theta(y)).
\] 
We then have
\begin{eqnarray*}
C_t(G)(\Theta(x),\Theta(y),\Theta(z)) & = & G_t(G_t(\Theta(x),\Theta(y)),\Theta(z)) \\
& = & G_t(\Theta_t(F_t(x,y)),\Theta(z)) \\
& = & G_t(\Theta(F^{+}),\Theta(z))G_t(\Theta(F^{-}),\Theta(z)) \\
& = & \Theta_t(F(F^+,z))\Theta_t(F(F^-,z)) \\
& = & \Theta_t(F_t(x,y),z).
\end{eqnarray*}
\end{proof}

To conclude this section, we compute the composite functor
\[
W:\FGL \xrightarrow{N} (2-\FGL)_{I} \xrightarrow{C} \FTL
\]
by making the roots of $W(F)$ explicit. Let $R$ be a ring and let $F(x,y)$ be a commutative formal group law. In view of Remark \ref{rem:rootsN}, we may choose the roots of $N_t(F)$ to be $F^{[1]}=-F(x,y)F(\bar x,\bar y)=-F(x,y)\overline{F(x,y)}$ and $F^{[2]}=-F(\bar x,y)F(x,\bar y)=-F(\bar x,y)\overline{F(\bar x,y)}$ in $R[[x,y]]$. 

\begin{lm}
Under the finite ring extension $R[[-x\bar x,-y\bar y,-z\bar z]]\subset R[[x,y,z]]$, the roots of $W(F)$ are given by 
\[
W(F)^{[i]}=\begin{cases} 
-F(F(x,y),z)F(F(\bar x,\bar y),\bar z) & \text{ if $i=1$.} \\
-F(F(\bar x,y),z)F(F(x,\bar y),\bar z) & \text{ if $i=2$.} \\
-F(F(x,\bar y),z)F(F(\bar x,y),\bar z) & \text{ if $i=3$.} \\
-F(F(\bar x,\bar y),z)F(F(x,y),\bar z) & \text{ if $i=4$.}
\end{cases}
\]
\end{lm}

\begin{proof}
By construction, we have
\begin{eqnarray*}
N_1(F)(F^{[1]},-z\bar z) & = & N_1(F)(-F(x,y)\overline{F(x,y)},-z\bar z)  \\
& = & -F(F(x,y),z)F(F(\bar x,\bar y),\bar z) -F(F(\bar x,\bar y),z)F(F(x,y),\bar z)
\end{eqnarray*}
and
\[
N_2(F)(F^{[1]},-z\bar z)=F(F(x,y),z)F(F(\bar x,\bar y),\bar z)F(F(\bar x,\bar y),z)F(F(x,y),\bar z)
\] 
showing that $ -F(F(x,y),z)F(F(\bar x,\bar y),\bar z)$ and $-F(F(\bar x,\bar y),z)F(F(x,y),\bar z)$ are roots. We may compute $N_1(F)(F^{[2]},-z\bar z)$ and $N_2(F)(F^{[2]},-z\bar z)$ in a similar fashion and then deduce that the roots of 
\[
(1+N_1(F)(F^{[1]},-z\bar z)t+N_2(F)(F^{[1]},-z\bar z)t^2)(1+N_1(F)(F^{[2]},-z\bar z)t+N_2(F)(F^{[2]},-z\bar z)t^2)
\]
are as claimed.
\end{proof}

\section{Example: symplectic oriented ring spectra}\label{num:geometricFTL}\label{sec:geometry}

The first two authors have shown that Borel classes for $\SL$-oriented motivic spectra yield formal ternary laws, see \cite[section 2]{DF21}. We now explain how this result may be extended to $\Sp$-oriented theories.

Let $\E$ be a motivic ring spectrum, that is a (weakly) commutative monoid in the stable
$\AA^1$-homotopy category over $S$, where is is a noetherian scheme of finite Krull dimension. The motivic ring spectrum $\MSp$ has been introduced and studied by Panin and Walter \cite{PWcobord}. They defined and investigated {\em symplectic oriented} cohomology theories and ring spectra (or $\Sp$-oriented for short). The key result is \cite[Theorem 1.1]{PWcobord}, see also \cite[section 3]{An20} and the summary provided in \cite[section 2.1]{DF21}. 

\begin{rem}\label{rem:topologicalMsp}
The definition of the classical topological spectrum $MSp$ is well-known, see e.g. Switzer's book \cite{Sw75}. The graded ring $\pi_*(MSp) \otimes \ZZ[1/2]$ is isomorphic to the (known) polynomial algebra $MSO^*\otimes \ZZ[1/2]$ by work of Novikov, but integrally $\pi_*(MSp)$ is known only in degrees $\leq 100$ in which the only torsion is $2$-torsion. The first element of $4$-torsion appears in degree $103$. See e.g. \cite{St67}, \cite[Remark 20.36]{Sw75}, \cite[Theorem 5.25]{Ray72} and \cite[Section 8]{Kochman93}.
Over a general base field the bigraded motivic homotopy groups $\pi_{**}(\MSp)$ are not known, either, not even after tensoring with $\ZZ[1/2]$, but $\pi_{**}(\MSp)[\eta^{-1}]$ is known by the recent \cite[Corollary 1.3(2)]{BH} over fields, later extended to Dedekind domains by Bachmann \cite{Bachmann21}. One may show that $\MSp$ maps to $MSp^{top}$ under complex realization. We could not find a proof for this in the literature. One should use that $Sp_{2n}(\CC)^{an}$ has $Sp(n)$ as maximal compact subgroup, that Thom spaces realize to Thom spaces and that $Sp_{2n}$ is ``special'' in the sense of Serre and others. The following results also hold for symplectic oriented topological commutative ring spectra, for which we may set $\epsilon=-1$ everywhere.   
\end{rem}

The following result of the third author is crucial. Its proof will be published elsewhere.
\begin{thm}[Fasel]\label{thm:epsilonlinearity}
Let $R$ be a ring, let $U$ be the universal bundle on $\mathrm{BSp}_{2,R}$ and let $\lambda\in R^\times$ be an invertible element. Let further $\mathrm{th}(U,\varphi):\mathrm{Th}(U)\to \mathbf{MSp}(2)[4]$ be the canonical Thom class, where $\varphi:U\to U^\vee$ is the canonical symplectic form. Then
\[
\mathrm{th}(U,\langle\lambda\rangle \varphi)=\langle \lambda\rangle \mathrm{th}(U,\varphi).
\]

\end{thm}

\begin{cor}
For any $\Sp$-oriented motivic ring spectrum $\E$, the Borel classes yield a formal ternary law as in Definition \ref{df:FTL}.
\end{cor}
\begin{proof}
Commutativity and associativity are obvious. The above theorem for $\lambda=-1$ implies $\epsilon$-linearity for $\MSp$, hence by \cite[Theorem 1.1]{PWcobord} for all $\Sp$-oriented ring spectra. Now the remaining two axioms follow as in \cite[Proposition 2.3.4]{DF21}. 
\end{proof}

\appendix
\section{Sage implementation of FTL computations}\label{num:sage}
We describe here the algorithm used to compute a standard (Groebner) basis of the ideal $\mathcal R_{FTL}$ (see Proposition \ref{prop:universal_FTL}) of and its implementation using the Sagemath \cite{sagemath} software together with the libsingular \cite{singular} back-end for multivariate polynomial computations. 

\subsection{Main problem characteristics and design principle}
The inputs of the algorithm are a base numerical ring $R$ (among $\mathbb{Z},\mathbb{Q}, \mathbb{Z}/n\mathbb{Z}$), and a degree range $(d_{min}, d_{max}) \in \mathbb{Z}^2, -4 \le d_{min} \le d_{max}$. With those restrictions, the unknown FTL coefficients span a finite set $a =  \lbrace a^{l}_{ijk}, d_{min} \le deg(a^{l}_{ijk}) \le d_{max}, 1,  \le l \le 4 \rbrace $. The five axioms stated in Definition \ref{df:FTL} induce relations between the FTL coefficients in the polynomial ring $R[\epsilon, a, \alpha]$, where $\alpha$ stands for one or more optional formal variables. For technical reasons this is more convenient than to work explicitly in $R_{\epsilon}[a]$. Additional constraints can be appended to the relation set in order to obtain the desired quotient ring, (for instance $\epsilon^2 -1$ for $R_{\epsilon}[a]$ or  $2 \alpha -1 = 0$ for  $R[\epsilon,a][\frac{1}{2}]$). 
The overall purpose of the algorithm is first to compute the finite set $S_{FTL} = \cup S_{axiom} = \lbrace r_p, 1\le p \le N , r_p \in R[\epsilon, a] \rbrace $ of relations stemming from the five FTL axioms and then compute a standard basis $g_{FTL}$ of the ideal generated by $S_{FTL}$. Noting $Gb$ the algorithm used for Groebner basis computation, a direct resolution scheme would simply read $g_{FTL} \leftarrow Gb(\mathcal{I}(S_{FTL}))$. \\

Unlike symmetry, (weak) neutral element, and $\epsilon$-linearity axioms, the set of relations induced by the associativity axiom tends to be quite large (hundreds to hundreds of thousands), with total degrees up to a few tens even for relatively small values ($1$ or $2$) of $d_{max}$. 
Even using a sparse storage scheme, the memory footprint of the full relations set can easily grow to tens of gigabytes. A second issue is the Groebner basis computation, which can generate numerous intermediate $S$-polynomials with huge coefficients. Either of memory and processing time costs can easily saturate available resources, even on "fat" computing nodes with hundreds of gigabytes of memory.

\subsection{Description of the implemented scheme}
The resolution scheme presented hereafter is designed to mitigate those technical difficulties and allow to solve the problem for a limited range of degrees (up to $d_{max}=2$ for some special cases).
Two optimization strategies are used to limit resources consumption. The first is to express the problem in the most compact manner to avoid representation overhead costs. The second is divide-and-conquer, i.e. split the whole problem into a sequence of smaller ones. \\

A first representation optimization stems from the symmetry axiom, which is the easiest to deal with. The symmetry relation set is fully generated by simple relations of the form $a^{l}_{ijk} - a^{l}_{\sigma(ijk)} = 0$, with $\sigma \in S_3$. For any fixed $l$, FTL coefficients with the same $l$ value and whose $(i',j',k')$ indexes share a common orbit under $S_3$ can be replaced by a single coefficient with indices $(i_0,j_0,k_0)$, which amounts to simply removing them and replacing the monomial $X^i_0 Y^j_0Z^k_0$ by the corresponding monomial symmetric polynomial in the expression of $F_t$. We note $\overline{a}$ this reduced set of FTL coefficients. Working with relations in $R[\epsilon, \overline{a}]$ with this form for $F$, the symmetry axiom is structurally enforced, with a significant reduction in the number of variables. 
\\
Problem splitting is obtained by partitioning $S_{FTL}$ into $K$ small subsets $S_{FTL} = \cup_{p \le K} s_p$. In order to avoid storing the full set of relations, the scheme operates in an iterative manner on an ideal sequence $\mathcal{I}_p$, defined by its generating set $g_p$, with initial value $g_0 = \emptyset$.  A each step $ p \le K$ of the algorithm, the small subset $s_p$ of the relations is generated on the fly. In the current implementation, $s_p$ is the full set of relations for all axioms except associativity, and a single relation for each of the associativity axiom relations. \\ 
A reduced Groebner basis of $\mathcal{I}(g_p \cup s_p)$ is then computed and used to replace $g$. A single  iteration of the scheme at step $p$ reads thus $g_{p+1} \leftarrow Gb(\mathcal{I}(g_{p} \cup s_p))$. Obviously, this scheme is more efficient memory-wise than the naive scheme provided that the intermediate generators sets $g$ size and complexity stay lower than the one of the full set of relations. While this is true for the final generating set $g_{K}$ ($S_{FTL}$ has a lot of redundancy), this is not guaranteed for the intermediate ones, whose complexity depend both on the way $S_{FTL}$ is partitioned, and the underlying Groebner basis algorithm. From a computation cost point of view, the scheme is obviously not optimally efficient : indeed the GB algorithm is only providing partial information about the ideal at each step, preventing it from operating optimally and leading to quite a lot of redundant computations. In practice, in many cases the intermediate $g_p$ complexity stays rather low, is quite often  invariant $g_{p+1}=g_p$ due to the large redundancy of associativity relations.\\

Based on this observation, a slight variation of the iteration scheme has also been implemented : after generation of $s_p$, $s_p$ is reduced by $\mathcal{I}$ before attempting computation of the Groebner basis. The scheme thus reads $g_{p+1} \leftarrow Gb(\mathcal{I}(g_{p} \cup s_p /\mathcal{I}(g_p)))$. In cases when $g_{p+1} = g_{p}$, $s_p$ actually  reduces to $0$ so that there is no need for an additional Groebner basis computation. Though this pre-reduction step is formally redundant, it appears empirically to ease the work of the Groebner algorithm when working over $\mathbb{Z}$ and limit coefficient swelling.   
\\
A the end of each step, a second use of the representation sparsity optimization is then performed. Once the new set of generators $g_{p+1}$ is computed, trivial relations of the form $r_u = a^{l}_{ijk} - c = 0, c \in R[\epsilon]$ are searched for. Whenever such a relation is found, the corresponding FTL coefficient is replaced by the unique solution $c$, and $a^{l}_{ijk}$ removed from the pool of active variables of the working ring. While this is formally completely equivalent to keeping $r_u$ in the generating set $g_p$, it has a non-negligible impact on memory consumption and computation complexity. A first guaranteed reduction is due to the fact that the number of active variables drives the size of all multiindexes representing monomial exponents in the polynomial representation of elements of $R[\epsilon,\overline{a}]$). This implies memory savings proportional to the number of variables that can be removed that way, at the risk of a slight increase in the complexity of the coefficients. Second, in practice, the solutions $c$ are usually fairly simple and often zero, leading in a drastic reduction of the complexity of generating relations in the restricted polynomial ring. Most such substitutions stem from the neutral, weak neutral element and linearity axioms, which are treated first in the resolution scheme, and occur early in the computation.\\

\subsection{Short description of the code repository}
The the core algorithm and data structures are regrouped in a single Python module (ftl\_comp.py) meant to be imported and used from Sage. Basic usage examples are provided in the form of Jupyter notebooks for light computations and post-computation work using saved results, and batch python scripts for non-interactive heavier computations. A selection of computations results saved in json format is also provided.

\subsection{Sample results 1 : Computation of $\Walt^{\leq 0}[\frac 12]$. 
}\label{appA:sample}
The following results were obtained in the polynomial ring $\mathbb{Z}[\epsilon, a]$, with the additional relations $ 2 \alpha-1 = 0$ and $\epsilon^2-1 = 0$, which is equivalent to working in $\mathbb{Z}_{\epsilon}[\frac{1}{2}][a]$. For this range of degrees ($-4$ to $0$) there are $24$ FTL coefficients. 
For $23$ of those, fixed values solutions were found, 
\begin{equation*}
\left\lbrace
\begin{array}{l}
a^{4}_{000} = a^{3}_{000} = a^{4}_{100} = a^{2}_{000} = a^{3}_{100} = a^{4}_{200} = a^{1}_{000} = 0 \\
a^{2}_{100} = a^{3}_{200} = a^{4}_{300} = a^{4}_{110} = a^{4}_{210} = a^{3}_{110} = a^{4}_{111} = 0 \\
a^{4}_{400} = 1 \\
a^{2}_{200} = a^{4}_{220} = -4 \epsilon + 2 \\
a^{3}_{210} = a^{4}_{310} = 2 \epsilon - 2 \\
a^{1}_{100} = a^{3}_{300} = a^{2}_{110} = a^{4}_{211} = -2 \epsilon + 2
\end{array}
\right.
,
\end{equation*}
and $a^{3}_{111}$ is the ony  non-fixed FTL coefficient, with  the remaining relations
\begin{equation*}
\left\lbrace
\begin{array}{l@{\ = \ }c}
(a^{3}_{111})^{2} + 768 \epsilon - 832 & 0 \\
(a^{3}_{111} - 40) (\epsilon - 1) & 0\\
\end{array}
\right.
.
\end{equation*}
With the same parameters, removing the $2$-invertibility assumption yields the solution for $\Walt^{\leq 0}$. In that case only $15$ FTL coefficients  are fixed, $9$ remain free, with $134$ remaining relations. They are too numerous to list here, but can be recomputed or reloaded from the saved runs provided in the code repository: 

\url{https://plmlab.math.cnrs.fr/coulette/ftlcomp}. 

\subsection{Example 2 : Computation of $\Walt^{\leq 2}[\frac{1}{2}] / (\epsilon-1)$}\label{sec:2inverted}
While the computation of $\Walt^{\leq 2}$ and $\Walt^{\leq 2}[\frac{1}{2}]$ have proven untractable so far, a fast and quite simple solution can be obtained when $\epsilon = 1$. In that case $52$ of the $57$ coefficients have fixed values
\begin{equation*}
\left\lbrace
\begin{array}{l}
a^{4}_{000} = a^{3}_{000} = a^{4}_{100} = a^{2}_{000} = a^{3}_{100} = a^{4}_{200} = a^{1}_{000} = 0 \\
a^{2}_{100} = a^{3}_{200} = a^{4}_{300} = a^{1}_{100} = a^{3}_{300} = a^{1}_{200} = a^{2}_{300} = 0 \\
a^{3}_{400} = a^{4}_{500} = a^{1}_{300} = a^{2}_{400} = a^{3}_{500} = a^{4}_{600} = a^{4}_{110} = 0 \\
a^{4}_{210} = a^{3}_{110} = a^{4}_{111} = a^{2}_{110} = a^{3}_{210} = a^{4}_{211} = a^{4}_{310} = 0 \\
a^{1}_{110} = a^{2}_{111} = a^{2}_{210} = a^{3}_{211} = a^{3}_{220} = a^{3}_{310} = a^{4}_{221} = 0 \\
a^{4}_{311} = a^{4}_{320} = a^{4}_{410} = a^{1}_{210} = a^{2}_{211} = a^{2}_{310} = a^{3}_{221} = 0 \\
a^{3}_{320} = a^{3}_{410} = a^{4}_{321} = a^{4}_{330} = a^{4}_{411} = a^{4}_{420} = a^{4}_{510} = 0 \\
a^{4}_{400} = 1 \\
a^{2}_{200} = a^{4}_{220} = -2
\end{array}
\right.,
\end{equation*}
with the five coefficients $a^{3}_{111}, a^{1}_{111}, a^{2}_{220},a^{3}_{311}, a^{4}_{222}$ and the remaining constraints
\begin{equation*}
\left\lbrace
\begin{array}{l@{ = \quad} l}
(a^{3}_{111})^{2} - 64& 0  \\
a^{3}_{111} a^{3}_{311} + 8 a^{4}_{222}& 0  \\
(a^{3}_{311})^{2} - (a^{4}_{222})^{2}& 0  \\
a^{3}_{111} a^{4}_{222} + 8 a^{3}_{311}& 0  \\
a^{1}_{111} - a^{3}_{311}& 0  \\
a^{2}_{220} - a^{4}_{222}& 0  \\
\end{array}
\right.
\end{equation*}

\bibliographystyle{amsalpha}
\bibliography{FTL2FGL}

\newcommand{\etalchar}[1]{$^{#1}$}
\providecommand{\bysame}{\leavevmode\hbox to3em{\hrulefill}\thinspace}
\providecommand{\MR}{\relax\ifhmode\unskip\space\fi MR }
\providecommand{\MRhref}[2]{%
  \href{http://www.ams.org/mathscinet-getitem?mr=#1}{#2}
}
\providecommand{\href}[2]{#2}
\begin{thebibliography}{BCD{\etalchar{+}}20}

\bibitem[AL75]{AL75}
J.~F. Adams and A.~Liulevicius, \emph{Buhstaber's work on two-valued formal
  groups}, Topology \textbf{14} (1975), no.~4, 291--296. \MR{391134}

\bibitem[Ana17]{An17}
A.~Ananyevskiy, \emph{Stable operations and cooperations in derived {W}itt
  theory with rational coefficients}, Ann. K-Theory \textbf{2} (2017), no.~4,
  517--560. \MR{3681106}

\bibitem[Ana20]{An20}
\bysame, \emph{S{L}-oriented cohomology theories}, Motivic homotopy theory and
  refined enumerative geometry, Contemp. Math., vol. 745, Amer. Math. Soc.,
  Providence, RI, [2020] \copyright 2020, pp.~1--19. \MR{4071210}

\bibitem[Bac21]{Bachmann21}
T.~Bachmann, \emph{{$\eta$-periodic motivic stable homotopy theory over Dedkind
  domains}}, arXiv:2006.02086, 2021.

\bibitem[BCD{\etalchar{+}}20]{BCDFO21}
T.~Bachmann, B.~Calm\`es, F.~D\'{e}glise, J.~Fasel, and P.A. {\O}stv\ae~r,
  \emph{{Milnor-Witt Motives}}, arXiv: 2004.06634, 2020.

\bibitem[BH20]{BH}
T.~Bachmann and M.~J. Hopkins, \emph{{$\eta$-periodic motivic stable homotopy
  theory over fields}}, arXiv: 2005.06778, 2020.

\bibitem[BM00]{Barge00}
Jean Barge and Fabien Morel, \emph{Groupe de {C}how des cycles orient\'es et
  classe d'{E}uler des fibr\'es vectoriels}, C. R. Acad. Sci. Paris S\'er. I
  Math. \textbf{330} (2000), no.~4, 287--290. \MR{MR1753295 (2000m:14004)}

\bibitem[BN71]{BN}
V.~M. Buh\v{s}taber and S.~P. Novikov, \emph{Formal groups, power systems and
  {A}dams operators}, Mat. Sb. (N.S.) \textbf{84(126)} (1971), 81--118, English
  translation: Math. USSR-Sb. 13 (1971), 80--116. \MR{0291159}

\bibitem[Buh75]{Bu75}
V.~M. Buh\v{s}taber, \emph{Two-valued formal groups. {A}lgebraic theory and
  applications to cobordism. {I}}, Izv. Akad. Nauk SSSR Ser. Mat. \textbf{39}
  (1975), no.~5, 1044--1064, 1219. \MR{0469927}

\bibitem[Buh76]{Bu76}
\bysame, \emph{Two-valued formal groups. {A}lgebraic theory and applications to
  cobordism. {II}}, Izv. Akad. Nauk SSSR Ser. Mat. \textbf{40} (1976), no.~2,
  289--325, 469. \MR{0480545}

\bibitem[Buh78]{Bu78}
\bysame, \emph{Topological applications of the theory of two-valued formal
  groups}, Izv. Akad. Nauk SSSR Ser. Mat. \textbf{42} (1978), no.~1, 130--184,
  215. \MR{0494160}

\bibitem[CPZ13]{CPZ13}
B~Calm\`es, V.~Petrov, and K.~Zainoulline, \emph{Invariants, torsion indices
  and oriented cohomology of complete flags}, Ann. Sci. \'{E}c. Norm.
  Sup\'{e}r. (4) \textbf{46} (2013), no.~3, 405--448 (2013). \MR{3099981}

\bibitem[CZZ15]{CZZ}
B.~Calm\`es, K.~Zainoulline, and C.~Zhong, \emph{Equivariant oriented
  cohomology of flag varieties}, Doc. Math. (2015), no.~Extra vol.: Alexander
  S. Merkurjev's sixtieth birthday, 113--144. \MR{3404378}

\bibitem[D{\'e}g18a]{Deg16}
F.~D{\'e}glise, \emph{Bivariant theories in motivic stable homotopy}, Doc.
  Math. \textbf{23} (2018), 997--1076.

\bibitem[D{\'e}g18b]{Deg18}
\bysame, \emph{Orientation theory in arithmetic geometry},
  {$K$}-{T}heory---{P}roceedings of the {I}nternational {C}olloquium, {M}umbai,
  2016, Hindustan Book Agency, New Delhi, 2018, pp.~239--347.

\bibitem[DF21]{DF21}
F.~{D{\'e}glise} and J.~{Fasel}, \emph{{The Borel character}}, Journal of the
  IMJ (2021), Firstview.

\bibitem[DGPS21]{singular}
W.~Decker, G.-M. Greuel, G.~Pfister, and H.~Sch\"onemann, \emph{{\sc Singular}
  {4-2-0.p3} --- {A} computer algebra system for polynomial computations},
  \url{http://www.singular.uni-kl.de}, 2021.

\bibitem[Fas20]{FaselCHW}
J.~Fasel, \emph{Lectures on {C}how-{W}itt groups}, Motivic homotopy theory and
  refined enumerative geometry, Contemp. Math., vol. 745, Amer. Math. Soc.,
  [Providence], RI, [2020] \copyright 2020, pp.~83--121.

\bibitem[FH20]{FH21}
J.~Fasel and O.~Haution, \emph{{The stable Adams operations on Hermitian
  K-theory}}, 2020.

\bibitem[Haz78]{Ha78}
M.~Hazewinkel, \emph{Formal groups and applications}, Pure and Applied
  Mathematics, vol.~78, Academic Press, Inc. [Harcourt Brace Jovanovich,
  Publishers], New York-London, 1978. \MR{506881}

\bibitem[HK11]{HK11}
J.~Hornbostel and V.~Kiritchenko, \emph{Schubert calculus for algebraic
  cobordism}, J. Reine Angew. Math. \textbf{656} (2011), 59--85. \MR{2818856}

\bibitem[Hor05]{Hornb}
J.~Hornbostel, \emph{{$A^1$}-representability of {H}ermitian {$K$}-theory and
  {W}itt groups}, Topology \textbf{44} (2005), no.~3, 661--687.

\bibitem[Hoy15]{Hoyois}
M.~Hoyois, \emph{From algebraic cobordism to motivic cohomology}, J. Reine
  Angew. Math. \textbf{702} (2015), 173--226.

\bibitem[Koc93]{Kochman93}
S.~Kochman, \emph{{Symplectic cobordism and the computation of stable stems}},
  Mem. Amer. Math. Soc. \textbf{104} (1993), no.~496, x+88.

\bibitem[Las63]{Lashof}
R.~Lashof, \emph{Poincar\'{e} duality and cobordism}, Trans. Amer. Math. Soc.
  \textbf{109} (1963), 257--277.

\bibitem[Laz55]{Lazard55}
M.~Lazard, \emph{{Sur les groupes de Lie formels \`a un param\`etre}}, Bull.
  Soc. Math. France \textbf{83} (1955), 251--274.

\bibitem[LM07]{LM07}
M.~Levine and F.~Morel, \emph{Algebraic cobordism}, Springer Monographs in
  Mathematics, Springer, Berlin, 2007. \MR{2286826}

\bibitem[Mer21]{MerkurjevOper}
A.~S. Merkurjev, \emph{Additive operations between connective {$K$}-theory and
  {C}how theory}, Adv. Math. \textbf{388} (2021), Paper No. 107911, 25.

\bibitem[MH73]{MilHus}
J.~W. {Milnor} and D.~H. {Husemoller}, \emph{{Symmetric bilinear forms}},
  vol.~73, Springer-Verlag, Berlin, 1973.

\bibitem[NS{\O }09]{NSO}
N.~Naumann, M.~Spitzweck, and P.~A. {\O }stv{\ae}r, \emph{Motivic {L}andweber
  exactness}, Doc. Math. \textbf{14} (2009), 551--593.

\bibitem[Pan04]{PaninRR}
I.~Panin, \emph{Riemann-{R}och theorems for oriented cohomology}, Axiomatic,
  enriched and motivic homotopy theory, NATO Sci. Ser. II Math. Phys. Chem.,
  vol. 131, Kluwer Acad. Publ., Dordrecht, 2004, pp.~261--333.

\bibitem[PW10]{Panin10pred}
I.~Panin and C.~Walter, \emph{{Quaternionic Grassmannians and Pontryagin
  classes in algebraic geometry}}, arXiv:1011.0649, 11 2010.

\bibitem[PW18]{PWcobord}
\bysame, \emph{{On the algebraic cobordism spectra $\mathbf{MSL}$ and
  $\mathbf{MSp}$}}, arXiv: 1011.0651, 2018.

\bibitem[PW19a]{Panin19b}
I.~{Panin} and C.~{Walter}, \emph{{On the motivic commutative ring spectrum
  $\mathbf{BO}$}}, St. Petersbg. Math. J. \textbf{30} (2019), no.~6, 933--972,
  arXiv: 1011.0650.

\bibitem[PW19b]{Panin19}
I.~Panin and C.~Walter, \emph{{On the relation of symplectic algebraic
  cobordism to Hermitian $K$-theory}}, Proc. Steklov Inst. Math. \textbf{307}
  (2019), 162--173.

\bibitem[Qui69]{Quillen}
D.~Quillen, \emph{On the formal group laws of unoriented and complex cobordism
  theory}, Bull. Amer. Math. Soc. \textbf{75} (1969), 1293--1298.

\bibitem[Rav86]{Ra03}
D.~C. Ravenel, \emph{{Complex cobordism and stable homotopy groups of
  spheres}}, Pure and Applied Mathematics, vol. 121, Academic Press, Inc.,
  Orlando, FL, 1986. \MR{860042}

\bibitem[Ray72]{Ray72}
N.~Ray, \emph{{The symplectic bordism ring}}, Proc. Cambridge Philos. Soc.
  \textbf{71} (1972), 271--282. \MR{290384}

\bibitem[Sch17]{Schlichting17}
Marco Schlichting, \emph{Hermitian {$K$}-theory, derived equivalences and
  {K}aroubi's fundamental theorem}, J. Pure Appl. Algebra \textbf{221} (2017),
  no.~7, 1729--1844, arXiv:1209.0848.

\bibitem[Smi06]{Smirnov}
A.~L. Smirnov, \emph{The {R}iemann-{R}och theorem for operations in the
  cohomology of algebraic varieties}, Algebra i Analiz \textbf{18} (2006),
  no.~5, 210--236.

\bibitem[Spi20]{SpitCobord}
M.~Spitzweck, \emph{Algebraic cobordism in mixed characteristic}, Homology
  Homotopy Appl. \textbf{22} (2020), no.~2, 91--103.

\bibitem[ST15]{ST}
M.~Schlichting and G.~S. Tripathi, \emph{Geometric models for higher
  {G}rothendieck-{W}itt groups in {$\Bbb{A}^1$}-homotopy theory}, Math. Ann.
  \textbf{362} (2015), no.~3-4, 1143--1167.

\bibitem[Sto67]{St67}
R.~E. Stong, \emph{{Some remarks on symplectic cobordism}}, Ann. of Math. (2)
  \textbf{86} (1967), 425--433. \MR{219079}

\bibitem[Str]{Strickl}
N.~Strickland, \emph{Formal groups},
  \url{http://neil-strickland.staff.shef.ac.uk/courses/formalgroups/fg.pdf}.

\bibitem[Swi75]{Sw75}
R.~M. Switzer, \emph{Algebraic topology---homotopy and homology},
  Springer-Verlag, New York-Heidelberg, 1975, Die Grundlehren der
  mathematischen Wissenschaften, Band 212. \MR{0385836}

\bibitem[SY14]{SchYok}
J.~Sch\"{u}rmann and S.~Yokura, \emph{Motivic bivariant characteristic
  classes}, Adv. Math. \textbf{250} (2014), 611--649.

\bibitem[{The}21]{sagemath}
{The Sage Developers}, \emph{{S}agemath, the {S}age {M}athematics {S}oftware
  {S}ystem ({V}ersion 9.4)}, 2021, {\tt https://www.sagemath.org}.

\bibitem[Vis19]{VishikOper}
Alexander Vishik, \emph{Stable and unstable operations in algebraic cobordism},
  Ann. Sci. \'{E}c. Norm. Sup\'{e}r. (4) \textbf{52} (2019), no.~3, 561--630.

\end{thebibliography}

\end{document}